\documentclass[11pt,a4paper]{article}

\usepackage{epsf,epsfig,amsfonts,amsgen,amsmath,amstext,amsbsy,amsopn,amsthm,cases,listings,color
}
\usepackage{ebezier,eepic}
\usepackage{color}
\usepackage{multirow}
\usepackage{epstopdf}
\usepackage{graphicx}
\usepackage{pgf,tikz}
\usepackage{mathrsfs}
\usepackage[marginal]{footmisc}
\usepackage{enumitem}
\usepackage{appendix}

\usepackage{pgf,tikz,pgfplots}
\usepackage{subfigure}
\pgfplotsset{compat=1.12}
\usepackage{mathrsfs}
\usetikzlibrary{arrows}

\definecolor{uuuuuu}{rgb}{0.27,0.27,0.27}
\definecolor{sqsqsq}{rgb}{0.1255,0.1255,0.1255}

\newtheorem{dfn}{Definition} [section]
\newtheorem{thm}[dfn]{Theorem}
\newtheorem{lemma}[dfn]{Lemma}

\newtheorem{coro}[dfn]{Corollary}
\newtheorem{conj}[dfn]{Conjecture}
\newtheorem{claim}[dfn]{Claim}

\newtheorem{fact}[dfn]{Fact}
\newenvironment{pf}{\noindent{\bf Proof}}

\def\lc{\left\lceil}

\def\rc{\right\rceil}

\def\lf{\left\lfloor}

\def\rf{\right\rfloor}
\begin{document}
\title{\bf\Large On a generalized Erd\H{o}s-Rademacher problem}

\date{\today}

\author{Xizhi Liu\thanks{Department of Mathematics, Statistics, and Computer Science, University of Illinois, Chicago, IL, 60607 USA. Email: xliu246@uic.edu.
Research partially supported by NSF award DMS-1763317.}
\and
Dhruv Mubayi\thanks{Department of Mathematics, Statistics, and Computer Science, University of Illinois, Chicago, IL, 60607 USA. Email: mubayi@uic.edu.
Research partially supported by NSF award DMS-1763317.}
}
\maketitle
%\footnote{footnote}
%%%%%%%%%%%%%%%%%%%%%%%%%%%%%%%%%%%%%%%%%%%%%%%%%
\begin{abstract}
The triangle covering number of a graph is the minimum number of vertices that hit all triangles. Given positive integers $s,t$ and an $n$-vertex graph $G$ with $\lfloor n^2/4 \rfloor +t$ edges and triangle covering number $s$, we determine (for large $n$) sharp bounds on the minimum number of triangles in $G$ and also describe the extremal constructions.  Similar results are proved for cliques of larger size and color critical graphs.

This extends classical work of Rademacher, Erd\H os, and Lov\'asz-Simonovits whose results apply only to $s \le t$. Our results also address two conjectures of Xiao and Katona. We prove one of them and give a counterexample and prove a modified version of the other conjecture.
\end{abstract}
%%%%%%%%%%%%%%%%%%%%%%%%%%%%%%%%%%%%%%%%%%%%%%%%%

\section{Introduction}
A classical result of Mantel \cite{M07} states that every graph on $n$ vertices with $\lf n^2/4 \rf+1$ edges
contains at least one copy of $K_3$.
Rademacher showed that there are actually at least $\lf n/2 \rf$ copies of $K_3$ in such graphs.
Later, Erd\H{o}s \cite{E62A,E62B} proved that if $t \le cn$ for some small constant $c >0$,
then every graph on $n$ vertices with $\lf n^2/4 \rf+t$ edges
contains at least $t \lf n/2 \rf$ copies of $K_3$.
Erd\H{o}s also conjectured that the same conclusion holds for all $t < n/2$.
Later, Lov\'{a}sz and Simonovits \cite{LS83} proved Erd\H{o}s' conjecture
and they also proved a similar result for $K_k$ with $k \ge 4$.
In \cite{DM10}, the second author extended their results by proving tight bounds on the number of
copies of color critical graphs in a graph with a prescribed number of vertices and edges.

Given a graph $G$ we use $V(G)$ to denote its vertex set and use $E(G)$ to denote its edge set.
Let $v(G) = |V(G)|$ and $e(G) = |E(G)|$.
Sometimes we abuse notation and let $G = E(G)$ and $|G|=e(G)$.
For a fixed graph $F$ let $N_{F}(G)$ denote the number of copies of $F$ in $G$.
The $F$-covering number $\tau_{F}(G)$ of $G$ is the minimum size of $S \subset V(G)$
such that every copy of $F$ in $G$ has at least one vertex in $S$.
If $F = K_{k}$, then we simply use $N_{k}(G)$ and $\tau_{k}(G)$ to denote $N_{K_k}(G)$ and $\tau_{K_k}(G)$, respectively.

The classical Erd\H{o}s-Rademacher problem is to determine the minimum value of $N_{F}(G)$ for graphs $G$ with
fixed number of vertices and edges.
Very recently, Xiao and Katona \cite{XK20} posed a generalized Erd\H{o}s-Rademacher problem by putting constraints on $\tau_{F}(G)$.
More precisely,  they asked for the minimum value of $N_{F}(G)$ for graphs $G$ with a fixed number of vertices and edges
and a fixed $F$-covering number.
In particular, they proved that every graph $G$ on $n$ vertices with $\lf n^2/4 \rf+1$ edges and $\tau_{3}(G)=2$
must contain at least $n-2$ copies of $K_3$.
They also posed several conjectures for the general case.

\begin{conj}[Xiao and Katona, \cite{XK20}]\label{conj-1-K3}
Let $s > t \ge 1$ be fixed integers and let $n \ge n_0 = n_0(s,t)$ be sufficiently large.
Then every graph $G$ on $n$ vertices with $\lf n^2/4 \rf+t$ edges and $\tau_{3}(G) \ge s$
contains at least $(s-1)\lf n/2\rf + \lc n/2 \rc - 2(s-t)$ copies of $K_3$.
\end{conj}

Let $V$ be a set of size $n$.
Then a partition $V = V_1\cup \cdots \cup V_{k-1}$ is called balanced if
$\lc n/(k-1) \rc \ge |V_i| \ge \lf n/(k-1)\rf$ for all $i \in [k-1]$.

\medskip

\begin{conj}[Xiao and Katona, \cite{XK20}]\label{conj-2-Kk}
Let $s > t \ge 1, k \ge 4$ be fixed integers.
Then every graph $G$ on $n$ vertices with $t_{k-1}(n)+1$ edges and $\tau_{k}(G) \ge 2$
contains at least $\left(|V_1|+|V_2|-2\right) \prod_{i=3}^{k-1}|V_i|$ copies of $K_k$,
where $V_1\cup \cdots \cup V_{k-1}$ is a balanced partition of $[n]$
with $|V_1| \ge \cdots \ge |V_{k-1}|$.
\end{conj}

Xiao and Katona claimed that there is a common generalization of Conjectures \ref{conj-1-K3} and \ref{conj-2-Kk} without writing it explicitly. They also observed that the case $s \le t$ of these questions is a consequence of the previously mentioned results of Rademacher, Erd\H os~\cite{E62A, E62B} and Lovasz-Simonovits~\cite{LS83}.   It therefore suffices to consider only the case $s>t$ for these questions.

We show that Conjecture \ref{conj-1-K3} is not true in general and  give the correct bound
on the number of copies of $K_3$ for all $s,t$ and sufficiently large $n$.
On the other hand, we  prove Conjecture \ref{conj-2-Kk} for sufficiently large $n$
and we  also prove several generalizations of Conjecture \ref{conj-2-Kk} for
graphs $G$ with $t_{k-1}(n)+t$ edges and $\tau_{k}(G) \ge s$.
Our method also gives a bound, which is tight up to a smaller order error term,
for the number of color critical graphs $F$ in a graph with a fixed number of vertices and edges
and a fixed $F$-covering number.

\subsection{Triangles}
Let $\mathbb N =\{0,1,\ldots\}$ be the set of nonnegative integers.
For $s > t \ge 1$ and $n \in \mathbb{N}$ let $e(n) = n^2-4t_2(n) = n^2-4\lfloor n^2/4\rfloor \in \{0,1\}$ and
\begin{align}
M_{s,t} = M_{s,t}(n) =\left\{m\in \mathbb{N}: \left(4s-4t-4m+e(n)\right)^{1/2} \in \mathbb{N}\right\}. \notag
\end{align}
Note that $M_{s,t} \ne \emptyset$ since $s-t \in M_{s,t}$. Let
%\begin{align}
%m_{s,t} = \min\left\{m\in \mathbb{N}: \sqrt{n^2-4t_{2}(n)+4s-4t-4m} \in \mathbb{N}\right\}. \notag
%\end{align}
\begin{align}
m_{s,t} = m_{s,t}(n) =\min \, M_{s,t}, \notag
\end{align}
and let $$R_{3}(n,s,t) = \left(4s-4t-4m_{s,t}+e(n)\right)^{1/2} \in \mathbb N.$$

Define
\begin{align}
n^{+}_{s,t}  = \frac{1}{2}\left(n + R_{3}(n,s,t) \right) \quad {\rm and} \quad
n^{-}_{s,t}  = \frac{1}{2}\left(n - R_{3}(n,s,t) \right). \notag
\end{align}

Let $B_{s,t}(n)$ be the complete bipartite graph on $n$ vertices with two parts $V_1$ and $V_2$
such that $|V_1|= n^{+}_{s,t}$ and $|V_2| = n^{-}_{s,t}$.

Let $\mathcal{BM}_{s,t}(n)$ consist of all graphs  obtained from $B_{s,t}(n)$ as follows:
take  distinct vertices $u_1,\ldots,u_s,v_1,\ldots,v_s$ in $V_1$,
add the edges $u_1v_1, \ldots, u_{s}v_{s}$ and remove $m_{s,t}$ distinct edges
$e_{1},\ldots,e_{m_{s,t}}$ such that every $e_i$ has one endpoint in
$\{u_1,\ldots, u_s,v_1,\ldots,v_s\}$ and the other endpoint in $V_2$,
and there is no triangle with three edges in
$\{e_{1},\ldots,e_{m_{s,t}}, u_1v_1, \ldots, u_{s}v_{s}\}$.

Let $\mathcal{BS}_{s,t}(n)$ consists of all graphs  obtained from $B_{s,t}(n)$ as follows:
take distinct vertices $u'_1,\ldots,u'_{s-1},v'_1,\ldots,v'_{s-1}$ in $V_1$
and distinct vertices  $u'_s,v'_s$ in $V_2$,
add the edges $u'_1v'_1, \ldots, u'_{s}v'_{s}$ and remove
$m_{s,t}$ distinct edges $e'_1,\ldots,e'_{m_{s,t}}$ such that
every $e'_i$ has one endpoint in $\{u'_1,\ldots,u'_{s-1},v'_1,\ldots,v'_{s-1}\}$
and the other endpoint in $\{u'_s,v'_s\}$
and there is no triangle with three edges in
$\{e'_1,\ldots,e'_{m_{s,t}},u'_1v'_1, \ldots, u'_{s}v'_{s}\}$.

We abuse notation by letting $BM_{s,t}(n)$ and $BS_{s,t}(n)$ denote a generic member of $\mathcal{BM}_{s,t}(n)$ and $\mathcal{BS}_{s,t}(n)$ respectively.

\begin{figure}[htbp]
\centering
\subfigure[$BM_{s,t}(n)$.]{
\begin{minipage}[t]{0.45\linewidth}
\centering
\begin{tikzpicture}[xscale=3,yscale=3]
    \node (u1) at (-0.6,0.5) {};
    \node (v1) at (-0.4,0.5) {};
    \node (w1) at (0.5,0.4) {};
    \node (u2) at (-0.6,0.3) {};
    \node (v2) at (-0.4,0.3) {};
    \node (w2) at (0.5,0.2) {};
    \node (c1) at (-0.5,0.1) {};
    \node (c2) at (-0.5,0.05) {};
    \node (c3) at (-0.5,0) {};
    \node (d1) at (0,0.1) {};
    \node (d2) at (0,0.05) {};
    \node (d3) at (0,0) {};
    \node (e1) at (0.5,0) {};
    \node (e2) at (0.5,-0.05) {};
    \node (e3) at (0.5,-0.1) {};
    \node (u3) at (-0.6,-0.1) {};
    \node (v3) at (-0.4,-0.1) {};
    \node (w3) at (0.5,-0.2) {};
    \node (c4) at (-0.5,-0.3) {};
    \node (c5) at (-0.5,-0.35) {};
    \node (c6) at (-0.5,-0.4) {};
    \node (u4) at (-0.6,-0.5) {};
    \node (v4) at (-0.4,-0.5) {};
    \fill (u1) circle (0.015) node [below]   {$u_1$};
    \fill (v1) circle (0.015) node [below]  {$v_1$};
    \fill (w1) circle (0.015) node [below]  {$w_1$};
    \fill (u2) circle (0.015) node [below]   {$u_2$};
    \fill (v2) circle (0.015) node [below]  {$v_2$};
    \fill (w2) circle (0.015) node [below]  {$w_2$};
    \fill (c1) circle (0.01);
    \fill (c2) circle (0.01);
    \fill (c3) circle (0.01);
    \fill (d1) circle (0.01);
    \fill (d2) circle (0.01);
    \fill (d3) circle (0.01);
    \fill (e1) circle (0.01);
    \fill (e2) circle (0.01);
    \fill (e3) circle (0.01);
    \fill (u3) circle (0.015) node at  (-0.65,-0.2)  {$u_{m_{s,t}}$};
    \fill (v3) circle (0.015) node at (-0.35,-0.2)  {$v_{m_{s,t}}$};
    \fill (w3) circle (0.015) node [below]  {$w_{m_{s,t}}$};
    \fill (c4) circle (0.01);
    \fill (c5) circle (0.01);
    \fill (c6) circle (0.01);
    \fill (u4) circle (0.015) node [below]   {$u_{s}$};
    \fill (v4) circle (0.015) node [below]  {$v_{s}$};

    \draw[rotate=0,line width=0.8pt] (-0.5,0) ellipse [x radius=0.4, y radius=0.7];
    \draw[rotate=0,line width=0.8pt] (0.5,0) ellipse [x radius=0.4, y radius=0.7];
    \draw[line width=0.8pt,color=sqsqsq,fill=sqsqsq,fill opacity=0.15] (-0.6,0.5) to (-0.4,0.5);
    \draw[line width=0.8pt,color=sqsqsq,fill=sqsqsq,fill opacity=0.15] (-0.6,0.3) to (-0.4,0.3);
    \draw[line width=0.8pt,color=sqsqsq,fill=sqsqsq,fill opacity=0.15] (-0.6,-0.1) to (-0.4,-0.1);
    \draw[line width=0.8pt,color=sqsqsq,fill=sqsqsq,fill opacity=0.15] (-0.6,-0.5) to (-0.4,-0.5);
    \draw[line width=0.8pt,color=sqsqsq,fill=sqsqsq,fill opacity=0.15, dash pattern=on 1pt off 1.2pt] (-0.4,0.5) to (0.5,0.4);
    \draw[line width=0.8pt,color=sqsqsq,fill=sqsqsq,fill opacity=0.15, dash pattern=on 1pt off 1.2pt] (-0.4,0.3) to (0.5,0.2);
    \draw[line width=0.8pt,color=sqsqsq,fill=sqsqsq,fill opacity=0.15, dash pattern=on 1pt off 1.2pt] (-0.4,-0.1) to (0.5,-0.2);

    \node at (-0.5,-0.7-0.1) {$V_1$};
    \node at (0.5,-0.7-0.1) {$V_2$};
\end{tikzpicture}
%\caption{The $3$-graph $\mathcal{G}^1$}
\end{minipage}
}
\hfill
\subfigure[$BM_{s,t}(n)$.]{
\begin{minipage}[t]{0.45\linewidth}
\centering
\begin{tikzpicture}[xscale=3,yscale=3]
    \node (u1) at (-0.6,0.4) {};
    \node (v1) at (-0.4,0.4) {};

    \node (w1) at (0.5,0.4) {};
    \node (w0) at (0.5,0.6) {};

    \node (u2) at (-0.6,0.3) {};
    \node (v2) at (-0.4,0.3) {};

    \node (w2) at (0.5,0.2) {};

    \node (c1) at (-0.5,0.2) {};
    \node (c2) at (-0.5,0.15) {};
    \node (c3) at (-0.5,0.1) {};

    \node (d1) at (0,0.1) {};
    \node (d2) at (0,0.05) {};
    \node (d3) at (0,0) {};

    \node (e1) at (0.5,0) {};
    \node (e2) at (0.5,-0.05) {};
    \node (e3) at (0.5,-0.1) {};

    \node (u3) at (-0.6,0) {};
    \node (v3) at (-0.4,0) {};

    \node (u5) at (-0.6,-0.1) {};
    \node (v5) at (-0.4,-0.1) {};

    \node (w3) at (0.5,-0.2) {};

    \node (w4) at (0.5,0.1) {};

    \node (c4) at (-0.5,-0.3) {};
    \node (c5) at (-0.5,-0.35) {};
    \node (c6) at (-0.5,-0.4) {};

    \node (u4) at (-0.6,-0.5) {};
    \node (v4) at (-0.4,-0.5) {};

    \fill (u1) circle (0.015); %node [below]   {$u_1$};
    \fill (v1) circle (0.015); %node [below]  {$v_1$};

    \fill (w1) circle (0.015); %node [below]  {$w_1$};
    \fill (w0) circle (0.015); %node [below]  {$w_0$};

    \fill (u2) circle (0.015); %node [below]   {$u_2$};
    \fill (v2) circle (0.015); %node [below]  {$v_2$};

    \fill (w2) circle (0.015); %node [below]  {$w_2$};

    \fill (c1) circle (0.01);
    \fill (c2) circle (0.01);
    \fill (c3) circle (0.01);

    \fill (d1) circle (0.01);
    \fill (d2) circle (0.01);
    \fill (d3) circle (0.01);

    \fill (e1) circle (0.01);
    \fill (e2) circle (0.01);
    \fill (e3) circle (0.01);

    \fill (u3) circle (0.015);
    \fill (v3) circle (0.015);

    \fill (u5) circle (0.015); %node at  (-0.65,-0.4)  {$u_{\ell}$};
    \fill (v5) circle (0.015); %node at (-0.35,-0.4)  {$v_{\ell}$};

    \fill (w3) circle (0.015); %node [below]  {$w_{\ell'}$};

    \fill (w4) circle (0.015); %node [below]  {$w_{\ell'}$};

    \fill (c4) circle (0.01);
    \fill (c5) circle (0.01);
    \fill (c6) circle (0.01);

    \fill (u4) circle (0.015); %node [below]   {$u_{s}$};
    \fill (v4) circle (0.015); % node [below]  {$v_{s}$};

    \draw[rotate=0,line width=0.8pt] (-0.5,0) ellipse [x radius=0.4, y radius=0.7];
    \draw[rotate=0,line width=0.8pt] (0.5,0) ellipse [x radius=0.4, y radius=0.7];
    \draw[line width=0.8pt,color=sqsqsq,fill=sqsqsq,fill opacity=0.15, dash pattern=on 1pt off 1.2pt] (-0.6,0.4) to (0.5,0.6);
    \draw[line width=0.8pt,color=sqsqsq,fill=sqsqsq,fill opacity=0.15] (-0.6,-0.1) to (-0.4,-0.1);
    \draw[line width=0.8pt,color=sqsqsq,fill=sqsqsq,fill opacity=0.15, dash pattern=on 1pt off 1.2pt] (-0.4,-0.1) to (0.5,-0.2);
    \draw[line width=0.8pt,color=sqsqsq,fill=sqsqsq,fill opacity=0.15] (-0.6,0.4) to (-0.4,0.4);
    \draw[line width=0.8pt,color=sqsqsq,fill=sqsqsq,fill opacity=0.15] (-0.6,0.3) to (-0.4,0.3);
    \draw[line width=0.8pt,color=sqsqsq,fill=sqsqsq,fill opacity=0.15] (-0.6,0) to (-0.4,0);
    \draw[line width=0.8pt,color=sqsqsq,fill=sqsqsq,fill opacity=0.15] (-0.6,-0.5) to (-0.4,-0.5);
    \draw[line width=0.8pt,color=sqsqsq,fill=sqsqsq,fill opacity=0.15, dash pattern=on 1pt off 1.2pt] (-0.4,0.4) to (0.5,0.4);
    \draw[line width=0.8pt,color=sqsqsq,fill=sqsqsq,fill opacity=0.15, dash pattern=on 1pt off 1.2pt] (-0.4,0.4) to (0.5,0.2);
    \draw[line width=0.8pt,color=sqsqsq,fill=sqsqsq,fill opacity=0.15, dash pattern=on 1pt off 1.2pt] (-0.4,0.3) to (0.5,0.1);
    \draw[line width=0.8pt,color=sqsqsq,fill=sqsqsq,fill opacity=0.15, dash pattern=on 1pt off 1.2pt] (-0.4,0) to (0.5,-0.2);

    \node at (-0.5,-0.7-0.1) {$V_1$};
    \node at (0.5,-0.7-0.1) {$V_2$};
\end{tikzpicture}
%\caption{The complement of $G^2$.}
\end{minipage}
}
\vfill
\subfigure[$BS_{s,t}(n)$.]{
\begin{minipage}[t]{0.45\linewidth}
\centering
\begin{tikzpicture}[xscale=3,yscale=3]
    \node (u1) at (-0.6,0.5) {};
    \node (v1) at (-0.4,0.5) {};
    \node (w1) at (0.4,0.4) {};
    \node (w2) at (0.6,0.4) {};
    \node (u2) at (-0.6,0.3) {};
    \node (v2) at (-0.4,0.3) {};
    \node (c4) at (-0.5,-0.25) {};
    \node (c5) at (-0.5,-0.3) {};
    \node (c6) at (-0.5,-0.35) {};
    \node (c1) at (-0.5,0.1) {};
    \node (c2) at (-0.5,0.05) {};
    \node (c3) at (-0.5,0) {};
    \node (d1) at (0,0.3) {};
    \node (d2) at (0,0.25) {};
    \node (d3) at (0,0.2) {};

    \node (u3) at (-0.6,-0.1) {};
    \node (v3) at (-0.4,-0.1) {};

    \node (u4) at (-0.6,-0.4) {};
    \node (v4) at (-0.4,-0.4) {};
    \fill (u1) circle (0.015) node [below]   {$u'_1$};
    \fill (v1) circle (0.015) node [below]  {$v'_1$};
    \fill (w1) circle (0.015) node [below]  {$v'_s$};
    \fill (u2) circle (0.015) node [below]   {$u'_2$};
    \fill (v2) circle (0.015) node [below]  {$v'_2$};
    \fill (w2) circle (0.015) node [below]  {$u'_s$};
    \fill (c1) circle (0.01);
    \fill (c2) circle (0.01);
    \fill (c3) circle (0.01);
    \fill (d1) circle (0.01);
    \fill (d2) circle (0.01);
    \fill (d3) circle (0.01);

    \fill (u3) circle (0.015) node at  (-0.65,-0.2)  {$u'_{m_{s,t}}$};
    \fill (v3) circle (0.015) node at (-0.35,-0.2)  {$v'_{m_{s,t}}$};
    \fill (c4) circle (0.01);
    \fill (c5) circle (0.01);
    \fill (c6) circle (0.01);
    \fill (u4) circle (0.015) node at (-0.62,-0.5)   {$u'_{s-1}$};
    \fill (v4) circle (0.015) node at (-0.38,-0.5)  {$v'_{s-1}$};

    \draw[rotate=0,line width=0.8pt] (-0.5,0) ellipse [x radius=0.4, y radius=0.7];
    \draw[rotate=0,line width=0.8pt] (0.5,0) ellipse [x radius=0.4, y radius=0.7];
    \draw[line width=0.8pt,color=sqsqsq,fill=sqsqsq,fill opacity=0.15] (-0.6,0.5) to (-0.4,0.5);
    \draw[line width=0.8pt,color=sqsqsq,fill=sqsqsq,fill opacity=0.15] (-0.6,0.3) to (-0.4,0.3);
    \draw[line width=0.8pt,color=sqsqsq,fill=sqsqsq,fill opacity=0.15] (-0.6,-0.1) to (-0.4,-0.1);
    \draw[line width=0.8pt,color=sqsqsq,fill=sqsqsq,fill opacity=0.15] (-0.6,-0.4) to (-0.4,-0.4);
    \draw[line width=0.8pt,color=sqsqsq,fill=sqsqsq,fill opacity=0.15] (0.4,0.4) to (0.6,0.4);
    \draw[line width=0.8pt,color=sqsqsq,fill=sqsqsq,fill opacity=0.15, dash pattern=on 1pt off 1.2pt] (-0.4,0.5) to (0.4,0.4);
    \draw[line width=0.8pt,color=sqsqsq,fill=sqsqsq,fill opacity=0.15, dash pattern=on 1pt off 1.2pt] (-0.4,0.3) to (0.4,0.4);
    \draw[line width=0.8pt,color=sqsqsq,fill=sqsqsq,fill opacity=0.15, dash pattern=on 1pt off 1.2pt] (-0.4,-0.1) to (0.4,0.4);

    \node at (-0.5,-0.7-0.1) {$V_1$};
    \node at (0.5,-0.7-0.1) {$V_2$};
\end{tikzpicture}
%\caption{The complement of $G^2$.}
\end{minipage}
}
\hfill
\subfigure[$BS_{s,t}(n)$.]{
\begin{minipage}[t]{0.45\linewidth}
\centering
\begin{tikzpicture}[xscale=3,yscale=3]
    \node (u1) at (-0.6,0.5) {};
    \node (v1) at (-0.4,0.5) {};

    \node (w1) at (0.4,0.4) {};
    \node (w2) at (0.6,0.4) {};

    \node (u2) at (-0.6,0.3) {};
    \node (v2) at (-0.4,0.3) {};

    \node (c4) at (-0.5,-0.2) {};
    \node (c5) at (-0.5,-0.25) {};
    \node (c6) at (-0.5,-0.3) {};

    \node (c1) at (-0.5,0.2) {};
    \node (c2) at (-0.5,0.15) {};
    \node (c3) at (-0.5,0.1) {};

    \node (d1) at (0,0.3) {};
    \node (d2) at (0,0.27) {};
    \node (d3) at (0,0.24) {};

    \node (u3) at (-0.6,-0.1) {};
    \node (v3) at (-0.4,-0.1) {};

    \node (u4) at (-0.6,-0.4) {};
    \node (v4) at (-0.4,-0.4) {};

    \fill (u1) circle (0.015);% node [below]   {$u'_1$};
    \fill (v1) circle (0.015);% node [below]  {$v'_1$};
    \fill (w1) circle (0.015);% node [above]  {$v'_s$};
    \fill (u2) circle (0.015);% node [below]   {$u'_2$};
    \fill (v2) circle (0.015);% node [below]  {$v'_2$};
    \fill (w2) circle (0.015);% node [above]  {$u'_s$};
    \fill (c1) circle (0.01);
    \fill (c2) circle (0.01);
    \fill (c3) circle (0.01);

    \fill (d1) circle (0.007);
    \fill (d2) circle (0.007);
    \fill (d3) circle (0.007);

    \fill (u3) circle (0.015);% node at  (-0.65,-0.2)  {$u'_{m_{s,t}}$};
    \fill (v3) circle (0.015);% node at (-0.35,-0.2)  {$v'_{m_{s,t}}$};

    \fill (c4) circle (0.01);
    \fill (c5) circle (0.01);
    \fill (c6) circle (0.01);

    \fill (u4) circle (0.015);% node at (-0.62,-0.5)   {$u'_{s-1}$};
    \fill (v4) circle (0.015);% node at (-0.38,-0.5)  {$v'_{s-1}$};

    \draw[rotate=0,line width=0.8pt] (-0.5,0) ellipse [x radius=0.4, y radius=0.7];
    \draw[rotate=0,line width=0.8pt] (0.5,0) ellipse [x radius=0.4, y radius=0.7];
    \draw[line width=0.8pt,color=sqsqsq,fill=sqsqsq,fill opacity=0.15] (-0.6,0.5) to (-0.4,0.5);
    \draw[line width=0.8pt,color=sqsqsq,fill=sqsqsq,fill opacity=0.15] (-0.6,0.3) to (-0.4,0.3);
    \draw[line width=0.8pt,color=sqsqsq,fill=sqsqsq,fill opacity=0.15] (-0.6,-0.1) to (-0.4,-0.1);
    \draw[line width=0.8pt,color=sqsqsq,fill=sqsqsq,fill opacity=0.15] (-0.6,-0.4) to (-0.4,-0.4);
    \draw[line width=0.8pt,color=sqsqsq,fill=sqsqsq,fill opacity=0.15] (0.4,0.4) to (0.6,0.4);
    \draw[line width=0.8pt,color=sqsqsq,fill=sqsqsq,fill opacity=0.15, dash pattern=on 1pt off 1.2pt] (-0.4,0.5) to (0.4,0.4);
    \draw[line width=0.8pt,color=sqsqsq,fill=sqsqsq,fill opacity=0.15, dash pattern=on 1pt off 1.2pt] (-0.4,0.3) to (0.4,0.4);
    \draw[line width=0.8pt,color=sqsqsq,fill=sqsqsq,fill opacity=0.15, dash pattern=on 1pt off 1.2pt] (-0.4,-0.1) to (0.6,0.4);
    \draw[line width=0.8pt,color=sqsqsq,fill=sqsqsq,fill opacity=0.15, dash pattern=on 1pt off 1.2pt] (-0.6,-0.1) to (0.4,0.4);

    \node at (-0.5,-0.7-0.1) {$V_1$};
    \node at (0.5,-0.7-0.1) {$V_2$};
\end{tikzpicture}
%\caption{The complement of $G^2$.}
\end{minipage}
}
\centering
\caption{Several examples of graphs in $\mathcal{BM}_{s,t}(n)$ and $\mathcal{BS}_{s,t}(n)$.}
\end{figure}
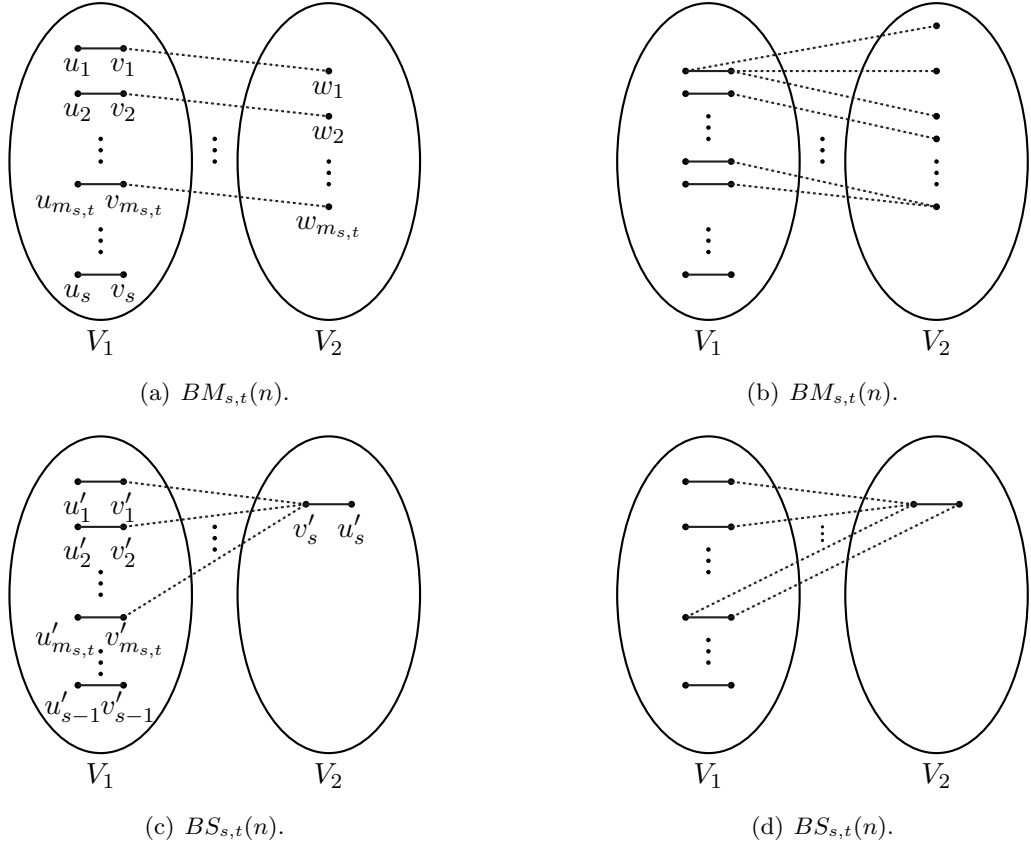

\begin{fact}
The following holds.
\begin{itemize}
\item $e\left(BM_{s,t}(n) \right) = e\left(BS_{s,t}(n) \right) = t_{2}(n)+t$.
\item $\tau_{3}\left(BM_{s,t}(n) \right) = \tau_{3}\left(BS_{s,t}(n) \right) = s$.
\item $N_{3}\left(BM_{s,t}(n) \right) = s \cdot n^{-}_{s,t} - m_{s,t}$.
\item $N_{3}\left(BS_{s,t}(n) \right) = (s-1)n^{-}_{s,t} + n^{+}_{s,t} - 2m_{s,t} = s \cdot n^{-}_{s,t} - m_{s,t} + (n^{+}_{s,t} - n^{-}_{s,t}-m_{s,t})$.
\end{itemize}
\end{fact}
By Lemma \ref{lemma-n+n-mst}, if for some $p\in \mathbb{N}$
\begin{align}
s-t =
\begin{cases}
p^2-1, & \text{ if $n$ is even},  \\
p(p+1)-1, & \text{ if $n$ is odd},
\end{cases}\notag
\end{align}
then $N_{3}\left(BM_{s,t}(n) \right) = N_{3}\left(BS_{s,t}(n) \right) = s \cdot n^{-}_{s,t} - m_{s,t}$.

Our first result shows that $BM_{s,t}(n)$ (and also $BS_{s,t}(n)$ for some special values of $s,t$)
contains the least number of copies of $K_3$
among all $n$-vertex graphs with $t_2(n)+t$ edges and $K_3$-covering number at least $s$.

\begin{thm}\label{thm-count-triangles}
Let $s > t \ge 1$. Then there exists $n_0 = n_0(s,t)$ such that the following holds for all $n\ge n_0$.
Let $G$ be a graph on $n$ vertices with $t_{2}(n)+t$ edges.
If $\tau_{3}(G) = s$, then
\begin{align}
N_{3}(G) \ge s \cdot n^{-}_{s,t} - m_{s,t} \notag
\end{align}
Moreover, equality holds only if $G \cong BM_{s,t}(n)$ or $G \cong BS_{s,t}(n)$ except
when $(s,t) \in \{(2,1),(3,1),(4,1)\}$ and $n$ is even, or $(s,t) \in \{(3,2),(4,1),(5,1),(6,1)\}$ and $n$ is odd. For these exceptional cases there are other examples showing that the bound is best possible.
\end{thm}

Note that Theorem \ref{thm-count-triangles} shows that Conjecture \ref{conj-1-K3} is not true in general.
For example, let $n$ be even, $(s-t)^{1/2} \in \mathbb{N}$ and $s-t > 4$.
Then
\begin{align}
N_{3}\left(BM_{s,t}(n)\right) =
s \cdot n^{-}_{s,t} - m_{s,t} =
s \cdot n^{-}_{s,t}
= \frac{sn}{2} - (s-t)^{1/2}s, \notag
\end{align}
which is strictly less that $sn/2-2(s-t)$.

%%%%%%%%%%%%%%%%%%%%%%%%%%%%%%
\subsection{$k$-cliques for $s = t+1$}
Let $V_1\cup \cdots \cup V_{k-1}$ be a partition of $[n]$ with $|V_1| \ge \cdots \ge |V_{k-1}|$.
Let $K[V_1,\ldots,V_{k-1}]$ be the complete $(k-1)$-partite graph on $[n]$
with parts $V_1,\ldots,V_{k-1}$.
If $V_1\cup \cdots \cup V_{k-1}$ is a balanced partition, then $K[V_1,\ldots,V_{k-1}]$ is called the Tur\'{a}n graph $T_{k-1}(n)$.
Let $t_{k-1}(n) = |T_{k-1}(n)|$.
The celebrated Tur\'{a}n theorem \cite{T41} states that the maximum number of edges
of an $n$-vertex $K_k$-free graph is uniquely achieved by $T_{k-1}(n)$.

For $s > m \ge 0$ and $\vec{x} = (x_1,\ldots,x_{k-1}) \in \mathbb{N}^{k-1}$ with $\sum_{i=1}^{k-1}x_i = n$
let $V_1\cup \cdots \cup V_{k-1}$ be a partition of $[n]$ with $|V_i| = x_i$ for $i \in [k-1]$.
Let $\mathcal{KM}_{m,s}(\vec{x})$ consist of all graphs that obtained from $K[V_1,\ldots,V_{k-1}]$ as follows:
take distinct vertices $u_1,\ldots,u_{s}$, $v_{1},\ldots,v_{s}$ in $V_1$,
%and (not necessarily distinct) vertices $w_1,\ldots,w_{m}$ in $V_{k-1}$,
add the edges $u_1v_1,\ldots, u_{s}v_{s}$ and remove $m$ distinct edges $e_1,\ldots, e_m$
such that every $e_i$ contains one vertex from $\{u_1,\ldots,u_{s}$, $v_{1},\ldots,v_{s}\}$
and one vertex from $V_{k-1}$ and there is no triangle with edges in
$\{e_1,\ldots, e_m, u_1v_1,\ldots, u_{s}v_{s}\}$.
We abuse notation by letting $KM_{s,t}(\vec{x})$ denote a generic member in $\mathcal{KM}_{m,s}(\vec{x})$.
%$\hat{v}_1w_1,\ldots, \hat{v}_{m}w_{m}$,
%where $\hat{v}_1,\ldots,\hat{v}_m \in \{v_1,\ldots,v_s\}$ are not necessarily distinct.
It is easy to see that
$$e\left(KM_{m,s}(\vec{x})\right) = \sum_{1\le i<j<k}x_ix_j + s-m \quad \hbox{and}\quad
N_{k}\left(KM_{m,s}(\vec{x})\right) = s\prod_{i=2}^{k-1}x_i - m \prod_{i=2}^{k-2}x_i.$$
Let us now consider some special cases of $KM_{m,s}(\vec{x})$ in more detail.

For $n \in \mathbb{N}$, write $$n= q_{n,k}(k-1) + r_{n,k} \qquad \hbox{where} \qquad 0 \le r_{n,k}<k-1.$$
 Writing $r=r_{n,k}$ and $q=q_{n,k}$,
let $\vec{y}_r \in \mathbb{N}^{k-1}$ be defined as follows:
\begin{align}
\vec{y}_r=
\begin{cases}
\left(q+1,q,\ldots,q,q-1\right)  & \text{if $r=0$}\\
\left(q+1,q,\ldots,q\right)
 & \text{if $r=1$}\\
  (q+2,\underbrace{q+1, \ldots, q+1}_{r-2 \text{ times}},\underbrace{q, \ldots, q}_{k-r \text{ times}})
  & \text{if $r\ge 2$}.
\end{cases}
\notag
\end{align}
Define
\begin{align}
N_{k}(n,s) =
\begin{cases}
s\cdot q^{k-3}\left(q-1\right) & \text{if $r=0$},\\
s\cdot q^{k-2} - q^{k-3} & \text{if $r=1$},\\
s\cdot \left(q+1\right)^{r-2}q^{k-r} & \text{if $r\ge 2$}.
\end{cases}
\notag
\end{align}
Observe that
$$e\left(KM_{0,s}(\vec{y}_r)\right) = e\left(KM_{1,s}(\vec{y}_1)\right) = t_{k-1}(n)+s-1 \quad \hbox{ for \,$r\ne 1$}$$
and
\begin{align}
N_{k}\left(KM_{m,s}(\vec{y}_{r})\right) = N_{k}(n,s) \notag
\end{align}
for $m = 0, r \neq 1$ and $m=1, r = 1$.

%
%
%When $r_{n,k} =0$,  $$
%N_{k}\left(KM_{m,s}(\vec{x})\right) = s\cdot q_{n,k}^{k-3}\left(q_{n,k}-1\right).$$
%
%When $r_{n,k} =1$,  $$ N_{k}\left(KM_{1,s}(\vec{x})\right) = s\cdot q^{k-2}_{n,k} - q^{k-3}_{n,k}.$$
%When $r_{n,k} \ge 2$, $$
%N_{k}\left(KM_{m,s}(\vec{x})\right) = s\cdot \left(q_{n,k}+1\right)^{r-2}q_{n,k}^{k-r}.$$

Our next result shows that the constructions defined above contain the least number of copies of $K_k$
in an $n$-vertex graph $G$ with $t_{k-1}(n)+s-1$ edges and $\tau_{k}(G) = s$.
%Let
%\begin{align}
%N_{k}(n,s) =
%\begin{cases}
% s\cdot q_{n,k}^{k-3}\left(q_{n,k}-1\right), & \text{if $r_{n,k}=0$},\\
%s\cdot q^{k-2}_{n,k} - q^{k-3}_{n,k}, & \text{if $r_{n,k}=1$},\\
%s\cdot \left(q_{n,k}+1\right)^{r-2}q_{n,k}^{k-r}, & \text{if $r_{n,k}\ge 2$}.
%\end{cases}
%\notag
%\end{align}

\begin{thm}\label{thm-count-Kk-s-t=1}
Let $k\ge 4$ and $s \ge 2$ be fixed integers.
Then there exists $n_1 = n_1(k,s)$ such that the following holds for all $n \ge n_1$.
Let $G$ be a graph on $n$ vertices with $t_{k-1}(n)+s-1$ edges.
If $\tau_{k}(G) = s$, then $N_{k}(G) \ge N_{k}(n,s)$.
Moreover, for $s \ge 3$ equality holds iff $G \cong KM_{0,s}(\vec{y}_{r_{n,k}})$ if $r_{n,k}\neq 1$
and $G \cong KM_{1,s}(\vec{y}_1)$ if $r_{n,k} = 1$.
\end{thm}

For $s\ge 2$, the following construction which was defined in \cite{XK20} also achieve the bound $N_{k}(n,2)$.
Let $V_1\cup \cdots \cup V_{k-1}$ be a balanced partition of $[n]$ with
$|V_1| \ge \cdots \ge |V_{k-1}|$.
Let $T_{k-1}^{\sqsubset}$ be obtained from $K[V_1,\ldots,V_{k-1}]$ as follows:
take two distinct vertices $u_1,v_1\in V_{1}$ and two distinct vertices $u_2,v_2\in V_2$,
add edges $u_1v_1,u_2v_2$ and remove the edge $v_1v_2$.
One can easily check that $N_{k}(T_{k-1}^{\sqsubset}) = \left(|V_1|+|V_2|-2\right)\prod_{i=3}^{k-1}|V_i| = N_{k}(n,s)$.
Therefore, Theorem \ref{thm-count-Kk-s-t=1} shows that Conjecture \ref{conj-2-Kk} is true for large $n$.

\subsection{$k$-cliques for large $s$}
Recall that for given $n$ and $k$, $q_{n,k} = \lf n/(k-1)\rf$ and $r_{n,k} = n - (k-1)q_{n,k}$.
Given $s > t \ge 1$ and $k\ge 3$,  let
\begin{align}
R_{k}(n,s,t) = \left(
\frac{2(k-1)(s-t)+(k-1-r_{n,k})r_{n,k}}
{k-2}
\right)^{1/2}. \notag
\end{align}
We note that while $R_k(n,s,t)$ depends on $n$ it is bounded from above by a function of only  $k,s,t$.
Let
$$n^{+}_{k,s,t} =  \frac{n+(k-2)R_{k}(n,s,t)}{k-1}\quad \hbox{ and  } \quad n^{-}_{k,s,t} = \frac{n-R_{k}(n,s,t)}{k-1}.$$
Suppose that $n^{-}_{k,s,t} \in \mathbb{N}$.
Then let $V_1 \cup \cdots \cup V_{k-1}$ be a partition of $[n]$ with $|V_1| = n^{+}_{k,s,t}$
and $|V_i| = n^{-}_{k,s,t}$ for $2 \le i \le k-1$.
Let $KM(n,k,s,t)$ be obtained from $K[V_1,\ldots,V_{k-1}]$ by taking distinct vertices $u_1,\ldots,u_s,v_1,\ldots,v_s$
in $V_1$ and then adding $u_1v_1,\ldots,u_sv_s$.
Using Lemma \ref{lemma-expression-turan-number} one can easily check that
$$e\left(KM(n,k,s,t)\right) = t_{k-1}(n)+t \quad \hbox{ and }
\quad N_{k}\left(KM(n,k,s,t)\right) = s \cdot (n^{-}_{k,s,t})^{k-2}.$$

The following result shows that if $s$ is large, then $KM(n,k,s,t)$ minimizes the number of copies of $K_k$
among all  $n$-vertex graphs $G$ with $t_{k-1}(n)+t$ edges and $\tau_{k}(G) = s$.

\begin{thm}\label{thm-count-Kk-s-t>1}
Let $s > t \ge 1$ and $k \ge 4$ be fixed integers.
There exists $n_{2} = n_2(k,s,t)$ such that the following holds for all $n\ge n_2$ and $s > 2R_{k}(n,s,t)$.
If $G$ is a graph on $n$ vertices with $t_{k-1}(n)+t$ edges and $\tau_{k}(G) = s$,
then $$N_{k}(G) \ge s \cdot (n^{-}_{k,s,t})^{k-2}.$$
Moreover, if $n^{-}_{k,s,t} \in \mathbb{N}$, then equality holds iff $G \cong KM(n,k,s,t)$.
\end{thm}

Note that we are not able to determine the exact minimum value of $N_{k}(G)$ for small $s$
because, similar to the situation in Theorem \ref{thm-count-triangles}, when $s$ is small there could be many constructions that achieve the minimum value of $N_{k}(G)$.
On the other hand, for the case $n^{-}_{k,s,t} \not\in \mathbb{N}$ our bound might be not tight and actually,
we think there might be a better bound for $N_{k}(G)$ in this case.

Let $R_{k}(s,t) = \left(2(k-1)(s-t)/(k-2)\right)^{1/2}$.
If $r_{n,k}=0$, then $R_{k}(n,s,t) = R_{k}(s,t)$.
Since $k \ge 4$ and $t\ge 1$, $s > 2R_{k}(s,t)$ holds for all $s \ge 11$.
Therefore, Theorem \ref{thm-count-Kk-s-t>1} gives the following corollary.

\begin{coro}\label{coro-n-mod-k-1=0}
Let $s > t \ge 1$ and $k \ge 4$ be fixed integers.
Suppose that $s\ge 11$.
Then there exists $n_{3} = n_3(k,s,t)$ such that the following holds for all $n\ge n_3$  and $n \equiv 0$ mod $k-1$.
If $G$ is a graph on $n$ vertices with $t_{k-1}(n)+t$ edges and $\tau_{k}(G) = s$,
then $N_{k}(G) \ge s \cdot (n^{-}_{k,s,t})^{k-2}$.
Moreover, if $n^{-}_{k,s,t} \in \mathbb{N}$, then equality holds iff $G \cong KM(n,k,s,t)$.
\end{coro}

After this work was done we found that similar results as in 
Theorems \ref{thm-count-triangles}, \ref{thm-count-Kk-s-t=1}, and \ref{thm-count-Kk-s-t>1}
were recently proved by Balogh and Clemen \cite{BC20}.

\subsection{Color critical graphs}
Given a graph $G$ let $\chi(G)$ denote the chromatic number of $G$.
Let $H$ be a subgraph of $G$.
Then the graph $G-H$ is obtained from $G$ by removing all edges that are contained in $G$.
In particular, if $e\in E(G)$, then $G-e$ is obtained from $G$ by removing $e$.

\begin{dfn}
Let $k \ge 3$.
A graph $F$ is $k$-critical if $\chi(F) = k$ and there exists $e\in E(F)$ such that $\chi(F-e)< k$.
\end{dfn}

Let $k \ge 3$ and let $F$ be a $k$-critical graph.
Let $c(n,F)$ denote the minimum number of copies of $F$ in the graph obtained from $T_{k-1}(n)$ by adding one edge.
The number $c(n,F)$ can be calculated using a formula in \cite{DM10} and in particular there exists a constant $\alpha_{F}>0$
depending only on $F$ such that $c(n,F) = \alpha_{F} n^{f-2} + \Theta(n^{f-3})$.

The second author proved~\cite{DM10}  that for any  $k$-critical graph $F$
there exists a constant $\delta = \delta_{F} > 0$ such that for every $1 \le t \le \delta n$
every $n$-vertex graph $G$ with $t_{k-1}(n)+t$ edges contains at least $t \cdot c(n,F)$ copies of $F$. We prove the analogous theorem for $\tau_F(G)=s$.

\begin{thm}\label{thm-count-critical-graphs}
Let $s > t \ge 1$ and $k \ge 3$ be fixed integers.
Let $F$ be a fixed $k$-critical graph on $f$ vertices.
Then there exists constants $C = C(F,s,t)$ and $n_{4} = n_4(F,s,t)$
such that the following holds for all $n\ge n_4$.
If $G$ is a graph on $n$ vertices with $t_{k-1}(n)+t$ edges and $\tau_{k}(G) = s$,
then $N_{F}(G) \ge s \cdot c(n,F) - C n^{f-3}$.
\end{thm}

This bound is tight up to an error term since the graph obtained from $T_{k-1}(n)$ by adding $s$ pairwise disjoint edges into
one part of $T_{k-1}(n)$ contains at most $s \cdot c(n,F)+ C' n^{f-3}$
copies of $F$ for some constant $C' >0$.

%%%%%%%%%%%%%%%%%%%%%%%%%%%%%%%%%%%%%%%%%%%%%%%%%%%%%
\section{Proofs}
\subsection{Lemmas}
In this section we prove several lemmas that will be used in our proofs.

\begin{dfn}
Let $k \ge 3$ and let $F$ be a $k$-critical graph.
Let $c(x_1,\ldots, x_{k-1},F)$ be the number of copies of $F$ in the graph
obtained from the complete $(k-1)$-partite graph with parts of sizes
$x_1,\ldots,x_{k-1}$ by adding one edge to the part of size $x_1$.
\end{dfn}

The following explicit expression for $t_{k-1}(n)$ is very useful in our calculations.

\begin{lemma}[e.g. see \cite{LS83}]\label{lemma-expression-turan-number}
Let $k\ge 3$ and suppose that $n \equiv r \mod (k-1)$ for some $0 \le r \le k-2$.
Then
\begin{align}
t_{k-1}(n) = \frac{(k-2)}{2(k-1)}n^2 - \frac{(k-1-r)r}{2(k-1)}. \notag
\end{align}
\end{lemma}

The following lemma gives a relation between $c(x_1,\ldots, x_{k-1},F)$ and $c(n,F)$.

\begin{lemma}[\cite{DM10}]\label{lemma-dhruv-c(n1,...,nk-1,F)}
Let $k\ge 3$ and $F$ be a $k$-critical graph.
Then there exists a constant $\gamma_{F} > 0$ depending only on $F$ such that the following holds for all sufficiently large $n$.
If $\sum_{i=1}^{k-1}x_i = n$ and $\lf n/(k-1)\rf-d \le x_i \le \lc n/(k-1)\rc+d$ for all $i\in [k-1] $and $d \le \frac{n}{3(k-1)}$,
then
\begin{align}
c(x_1,\ldots,x_{k-1},F) \ge c(n,F) - \gamma_{F} d n^{f-3}. \notag
\end{align}
\end{lemma}

The following lemma, which can be found in several places (e.g. see \cite{DM10}),
gives a bound on the size of each part for a $(k-1)$-partite graph whose number of edges is close to $t_{k-1}(n)$.

\begin{lemma}[e.g. see \cite{DM10}]\label{lemma-size-Vi}
Suppose that $k \ge 3$ is fixed, $n$ is sufficiently large, $d < n$ and $\sum_{i=1}^{k-1}x_i = n$.
If
\begin{align}
\sum_{1\le i < j \le k-1}x_ix_j \ge t_{k-1}(n)-d, \notag
\end{align}
then $\lf n/(k-1)\rf-d \le x_i \le \lc n/(k-1)\rc+d$ for all $i \in [k-1]$.
\end{lemma}

The following two results will be key in our proofs.

\begin{thm}[Graph removal lemma, e.g. see \cite{F11}]\label{thm-graph-remove-lemma}
Let $F$ be a graph with $f$ vertices.
Suppose that $G$ is a graph on $n$ vertices with $N_{F}(G) = o(n^{f})$.
Then one can remove $o(n^2)$ edges from $G$
such that the resulting graph is $F$-free.
\end{thm}

\begin{thm}[Erd\H{o}s-Simonovits stability theorem, \cite{S68}]\label{thm-stability}
Let $k\ge 3$ and $F$ be a $k$-critical graph.
Suppose that $G$ is an $F$-free graph on $n$ vertices with $t_{k-1}(n)-o(n^2)$ edges.
There $G$ can be made $(k-1)$-partite by removing $o(n^2)$ edges.
\end{thm}

Now we use the results above to obtain a rough structure of a graph with a fixed number of vertices and edges
and a fixed $F$-covering number that contains not many copies of $F$.

Given a graph $G$ and $v \in V(G)$ we use $N_{G}(v)$ to denote the neighbors of $v$ in $G$ and let $d_{G}(v) = |N_{G}(v)|$.
For a partition $V_1\cup \cdots \cup V_{k-1}$ of $V(G)$ we use $G[V_1,\ldots, V_{k-1}]$
to denote the induced $(k-1)$-partite subgraph of $G$ on $V_1\cup \cdots \cup V_{k-1}$.

\begin{lemma}\label{lemma-big-k-1-partite-subgp}
Let $s \ge 1,f \ge k\ge 3$ be fixed integers and $F$ be a fixed $k$-critical graph on $f$ vertices.
Then the following holds for sufficiently large $n$.
If $G$ is a graph on $n$ vertices with at least $t_{k-1}(n)+1$ edges and $N_{F}(G) \le (s+1/2) \cdot c(n,F)$,
then $G$ contains a $(k-1)$-partite subgraph $H$ such that $e(H) \ge e(G)-s$.
\end{lemma}
\begin{proof}
Let $\delta_1,\delta_2,\delta_3,\delta_4, \epsilon,\epsilon_1,\epsilon_2$ be constants such that
\begin{align}
0 < \delta_1 \ll \delta_2 \ll \delta_3 \ll \delta_4 \ll \epsilon_2 \ll \epsilon_1 \ll \epsilon \ll s^{-1}. \notag
\end{align}
Let $n$ be sufficiently large and in particular $n \gg s/\epsilon_2$.

Since $N_{F}(G) \le (s+1/2) \cdot c(n,F) < 2 s \alpha_{F} n^{f-2} = o(n^{f})$,
by the Graph removal lemma, we can remove at most $\delta_1 n^2$ edges from $G$
such that the resulting graph $G_1$ is $F$-free.
Since $e(G_1) \ge e(G)-\delta_1 n^2 > t_{k-1}(n)-\delta_1 n^2$,
by the Erd\H{o}s-Simonovits stability theorem,
$G_1$ contains a $(k-1)$-partite subgraph $G_2$ such that $e(G_2) \ge t_{k-1}(n)-\delta_2 n^2$.

Now let $H$ be a $(k-1)$-partite subgraph of $G$ with the maximum number of edges.
Then by the previous argument, $e(H) \ge e(G_2) \ge t_{k-1}(n)-\delta_2 n^2$.
Let $V_1\cup \cdots \cup V_{k-1}$ be a partition of $V(G)$ such that $H = G[V_1,\ldots, V_{k-1}]$
and let $x_i = |V_i|$ for $i\in[k-1]$.
An easy calculation shows that $\left| x_i - n/(k-1) \right| \le \delta_3 n$ for all $i\in[k-1]$.

Let $B$ denote the set of edges in $G$ that are contained inside $V_i$ for some $i\in[k-1]$,
i.e. $B = G - G[V_1,\ldots,V_{k-1}]$.
Let $M$ denote the set of pairs which intersect two parts that are not edges in $G$,
i.e. $M = K[V_1,\ldots, V_{k-1}] - G[V_1,\ldots, V_{k-1}]$.
Suppose that $|H| = t_{k-1}(n) - \ell$ for some $\ell \ge 0$.
Then $|M| \le \ell$ and $|B| \ge \ell + 1$.
For every $e\in B$ let $F(e)$ denote the number of copies of $F$ in $G$ containing the unique edge $e$ from $B$.
Let
\begin{align}
B_1 = \left\{e\in B: F(e) > (1-\epsilon)c(n,F)\right\} \notag
\end{align}
 and $B_2 = B \setminus B_1$.
A potential copy of $F$ is a copy of $F$ in $G \cup M$ that uses exactly one edge of $B$.

\begin{claim}\label{claim-lower-bound-B1}
$|B_1| \ge (1-\epsilon) |B|$.
\end{claim}
\begin{proof}[Proof of Claim \ref{claim-lower-bound-B1}]
Suppose that $|B_2| \ge \epsilon |B|$.
Let $e\in B_2$ and without loss of generality we may assume that $e \subset V_1$.
Then by Lemma \ref{lemma-dhruv-c(n1,...,nk-1,F)} the number of potential copies of $F$ containing $e$ is
\begin{align}
c(x_1,\ldots,x_{k-1},F) \ge c(n,F) - \gamma_{F} (\delta_3 n) n^{f-3} > (1-\delta_4) c(n,F). \notag
\end{align}

At least ${\epsilon} \cdot c(n,F)/2$ of these potential copies of $F$ have a pair from $M$, since otherwise
\begin{align}
F(e) \ge (1-\delta_4) c(n,F) - \frac{\epsilon}{2} c(n,F) > (1-\epsilon)c(n,F), \notag
\end{align}
a contradiction.
Now suppose that at least ${\epsilon} \cdot c(n,F)/4$ of these potential copies of $F$ have a pair from $M$ that does not intersect $e$.
For every $e' \in M$ with $e\cap e' = \emptyset$
the number of potential copies of $F$ in $G$ that contains both $e$ and $e'$ is at most $n^{f-4}$.
On the other hand, every potential copy of $F$ contains at most $f^2$ pairs from $M$.
Therefore,
\begin{align}
\frac{\epsilon}{4} c(n,F) \ge |M| f^2 n^{f-4}, \notag
\end{align}
which implies that
\begin{align}
\delta_2 n^2 \ge |M| \ge \frac{\frac{\epsilon}{4} c(n,F)}{f^2 n^{f-4}} > \frac{\epsilon \alpha_{F}}{8f^{2}}n^2, \notag
\end{align}
a contradiction. Here we used $|M| \le t_{k-1}(n) - e(H) \le \delta_2 n^2$.
Therefore, we may assume that at least ${\epsilon} \cdot c(n,F)/4$ of these potential copies of $F$ have a pair from $M$
which has nonempty intersection with $e$.
Similarly, since every $e'' \in M$ with $e''\cap e\neq \emptyset$ is contained in at most $n^{f-3}$ members in $F(e)$
and every potential copy of $F$ contains at most $f^2$ pairs from $M$,
the number of pairs from $M$ that has nonempty intersection with $e$ is at least
\begin{align}
\frac{\frac{\epsilon}{4} c(n,F)}{f^2 n^{k-3}} \ge \frac{\epsilon \alpha_F}{8f^{2}}n. \notag
\end{align}
Therefore, there exists $x \in e$ such that $d_{M}(x) \ge \frac{\epsilon \alpha_F}{16f^{2}}n$.

Let $A = \left\{ v \in V(G): d_{M}(v) \ge \frac{\epsilon \alpha_F}{16f^{2}}n\right\}$.
Since every $e\in B$ contains a vertex in $A$,
\begin{align}
\sum_{v\in A}d_{B_2}(v) \ge |B_2| \ge \epsilon |B| \ge \epsilon |M| \ge \frac{\epsilon}{2} \sum_{v \in A} d_{M}(v)
\ge \frac{\epsilon^2 \alpha_F}{32f^{2}}n|A|. \notag
\end{align}
Therefore, there exists $v \in A$ such that $d_{B_2}(v) \ge \frac{\epsilon^2 \alpha_F}{32f^{2}}n$ and without loss of generality
we may assume that $v \in V_1$.
Let $V_i' = N_{G}(v) \cap V_i$ for $i \in[k-1]$.
Then by the maximality of $H$ we have $|V_i'| \ge |V_{1}'| \ge \frac{\epsilon^2 \alpha_F}{32f^{2}}n$ for all $2 \le i \le k-1$.
Let $u \in V_1'$.
Then by Lemma \ref{lemma-dhruv-c(n1,...,nk-1,F)},
the number of potential copies of $F$ containing $uv$ in the complete $(k-1)$-partite graph $K[V_1',\ldots, V'_{k-1}]$
is at least
\begin{align}
c\left( |V_1'|,\ldots,|V_{k-1}'|,F \right)
\ge \frac{1}{2}\alpha_{F} \left(\frac{\epsilon^2 \alpha_F}{32f^{2}}n\right)^{k-2} \ge \epsilon_1 n^{k-2}. \notag
\end{align}
Summing over all $u\in V_1'$, there are at least
\begin{align}
\frac{\epsilon^2 \alpha_F}{32f^{2}}n \times \epsilon_1 n^{f-2} \ge \epsilon_2 n^{f-1} \ge  3s \cdot c(n,F) \notag
\end{align}
potential copies of $F$ containing $v$.
By the assumption that $N_{F}(G) \le (s+1/2) \cdot c(n,F)$,
at least half of these potential copies of $F$ must contain a pair from $M$,
and this pair cannot be incident with $v$, since $v$ is adjacent to all vertices in $\bigcup_{i=1}^{k-1}V_i'$.
Since the number of potential copies of $F$ that contain both $v$ and a pair from $M$ that is disjoint from $v$
is at most $n^{f-3}$ and each potential copy of $F$ contains at most $f^2$ pairs from $M$,
we obtain
\begin{align}
\delta_2 n^2 \ge |M| \ge \frac{\epsilon_2 n^{f-1}/2}{f^2 n^{f-3}} \ge \frac{\epsilon_2}{2f^2}n^2, \notag
\end{align}
a contradiction.
\end{proof}

\begin{claim}\label{claim-exact-size-B}
$|B| \le s$.
\end{claim}
\begin{proof}[Proof of Claim \ref{claim-exact-size-B}]
Suppose that $|B| \ge s+1$.
Then by Claim \ref{claim-lower-bound-B1},
\begin{align}
N_{F}(G)
\ge \sum_{e\in B_1}F(e)
 \ge   \sum_{e\in B_1} (1-\epsilon) c(n,F)
& \ge (1-\epsilon)^2 |B| c(n,F) \notag\\
& \ge (1-\epsilon)^2 (s+1) c(n,F)
> (s+1/2) \cdot c(n,F), \notag
\end{align}
a contradiction.
\end{proof}

Therefore, by Claim \ref{claim-exact-size-B}, $e(H) = e(G) - |B| \ge e(G) -s$.
This completes the proof of Lemma \ref{lemma-big-k-1-partite-subgp}.
\end{proof}

Now we use Lemma \ref{lemma-big-k-1-partite-subgp} to obtain a fine structure for
graphs with a fixed $F$-covering number and not many copies of $F$.

\begin{lemma}\label{lemma-G-G[V1,...Vk-1]-is-matching}
Let $f \ge k \ge 3, s > t \ge 1$ be fixed integers and $F$ be a fixed $k$-critical graph on $f$ vertices.
Then the following holds for sufficiently large $n$.
Let $G$ be a graph on $n$ vertices with $t_{k-1}(n)+t$ edges.
If $\tau_{F}(G) = s$ and $N_{F}(G) \le (s+1/2) \cdot c(n,F)$,
then there exists a partition $V(G) = V_1 \cup \cdots \cup V_{k-1}$
such that $G - G[V_1,\ldots,V_{k-1}]$ is a matching with $s$ edges.
\end{lemma}
\begin{proof}
Let $H$ be a $(k-1)$-partite subgraph of $G$ with the maximum number of edges and let $B = G - H$.
%Let  $V(G) = V_1 \cup \cdots \cup V_{k-1}$  be a partition such that $H = G[V_1,\ldots,V_{k-1}]$.
Since $N_{F}(G) \le (s+1/2) \cdot c(n,F)$, by Lemma \ref{lemma-big-k-1-partite-subgp}, $|B| \le s$.
So it suffice to show that $|B| \ge s$ and $B$ is a matching.

Let $\tau(B) = \min\left\{S \subset V(G): e \cap S \neq \emptyset, \forall e\in B\right\}$.
Since every copy of $F$ in $G$ must contain at least one edge in $B$, $\tau_{F}(G) \le \tau(B)$.
Therefore, $\tau(B) \ge s$.
Since $|B| \le s$, the only possibility is that $B$ is a matching of size $s$.
\end{proof}

\subsection{Proof of Theorem \ref{thm-count-triangles}}
In this section we prove Theorem \ref{thm-count-triangles}.
Recall that
for $s > t \ge 1$ and $n\in \mathbb{N}$
\begin{align}
n^{+}_{s,t}  = \frac{1}{2}\left(n + R_{3}(n,s,t) \right) \quad {\rm and} \quad
n^{-}_{s,t}  = \frac{1}{2}\left(n - R_{3}(n,s,t) \right), \notag
\end{align}
where
$R_{3}(n,s,t) = \left(4s-4t-4m_{s,t}+n^2-4t_{2}(n)\right)^{1/2}$
and
\begin{align}
m_{s,t} = \min\left\{m\in \mathbb{N}: \left(4s-4t-4m+n^2-4t_{2}(n)\right)^{1/2} \in \mathbb{N}\right\}. \notag
\end{align}
We will use the following lemma in our proof.

\begin{lemma}\label{lemma-n+n-mst}
Let $s> t \ge 1$ and $n\in \mathbb{N}$. Then
\begin{align}
n^{+}_{s,t} - n^{-}_{s,t} -m_{s,t} =
\begin{cases}
0 & \text{if $n$ is even and $s-t = p^2-1$ for some $p\in \mathbb{N}$}, \\
0 & \text{if $n$ is odd and $s-t = p(p+1)-1$ for some $p\in \mathbb{N}$}, \\
> 0 & \text{otherwise}.
\end{cases}\notag
\end{align}
\end{lemma}
\begin{proof}
First, notice that $n^{+}_{s,t} - n^{-}_{s,t} -m_{s,t} = \left(4s-4t-4m_{s,t}+n^2-4t_{2}(n)\right)^{1/2} - m_{s,t}$.

If $n$ is even, then $n^2 -4t_{2}(n) =0$.
Let $p\in \mathbb{N}$ be the largest integer such that $s-t = p^2 + q$ for some $q\in \mathbb{N}$.
Note that $q \le 2p$ since otherwise we would have $p^2 + q \ge (p+1)^2$, a contradiction.
Then $m_{s,t} = q$ and hence
\begin{align}
\left(4s-4t-4m_{s,t}+n^2-4t_{2}(n)\right)^{1/2} - m_{s,t} = 2p - m_{s,t} \ge 0 \notag
\end{align}
and equality holds iff $q = 2p$.

If $n$ is odd, then $n^2 -4t_{2}(n) =1$.
Let $p \in \mathbb{N}$ be the largest integer such that $s-t = p(p+1) + q$ for some $q\in \mathbb{N}$.
Note that $q \le 2p+1$ since otherwise we would have $p(p+1) + q \ge (p+1)(p+2)$, a contradiction.
Then $m_{s,t} = q$ and hence
\begin{align}
\left(4s-4t-4m_{s,t}+n^2-4t_{2}(n)\right)^{1/2} - m_{s,t} = 2p + 1 - m_{s,t} \ge 0 \notag
\end{align}
and equality holds iff $q = 2p+1$.
\end{proof}

Now we are ready to prove Theorem \ref{thm-count-triangles}.

\begin{proof}[Proof of Theorem \ref{thm-count-triangles}]
Let $s> t \ge 1$ be fixed and let $n$ be sufficiently large.
Let $G$ be a graph on $n$ vertices with $t_2(n)+t$ edges and $\tau_{3}(G) = s$.
Since $s\cdot n^{-}_{s,t} - m_{s,t}< (s+1/2)\cdot c(n,K_3)$,
we may assume that $N_{3}(G) \le (s+1/2)\cdot c(n,K_3)$.
So, by Lemma \ref{lemma-G-G[V1,...Vk-1]-is-matching}, there exists a partition $V(G) = V_1 \cup V_2$ such that
$B:= G - G[V_1,V_2]$ is a matching of size $s$.

Let $x = |V_1|$ and $y = |V_2|$ and note that $x+y = n$.
Without loss of generality we may assume that $x \ge y$.
Let $H = G[V_1,V_2]$, $M = K[V_1,V_2] - H$, and $m = |M|$.
Since $G-B = H = K[V_1,V_2]-M$,
we obtain $t_2(n) + t - s = xy - m = (n-y)y -m$.
Therefore, $m \in M_{s,t}$ and
\begin{align}
y = \frac{1}{2}\left(n - \left(4s-4t-4m+n^2-4t_{2}(n)\right)^{1/2}\right). \notag
\end{align}

Let $s_i = |B \cap \binom{V_i}{2}|$ for $i =1,2$ and note that $s_1 + s_2 = s$.
It is easy to see that the number of potential copies of $K_3$ is $s_1 y + s_2 x$.
We will consider two cases: either $s_i = s$ for some $i\in\{1,2\}$ or $s_1\ge 1$ and $s_2\ge 1$.

\textbf{Case 1:} $s_i = s$ for some $i\in\{1,2\}$.

We may assume that $s_2 = 0$ and the case $s_1 = 0$ can be solved using a similar argument.
Notice that for every $e \in M$ there is at most one potential copy of $K_3$ containing $e$.
Therefore,
\begin{align}
N_{3}(G)
\ge s y - m
= \frac{sn}{2} - \frac{s}{2} \left(4s-4t-4m+n^2-4t_{2}(n)\right)^{1/2} - m =: f(m). \notag
\end{align}
Then
\begin{align}
\frac{\mathrm{d} f(m)}{\mathrm{d} m} = \frac{s}{\left(4s-4t-4m+n^2-4t_{2}(n)\right)^{1/2}} -1. \notag
\end{align}
First let us assume that $s \ge 3$.
Then
\begin{align}
s^2 \ge 4s - 4 + 1 \ge 4s-4t-4m+n^2-4t_{2}(n). \notag
\end{align}
Therefore, $\frac{\mathrm{d} f(m)}{\mathrm{d} m} > 0$ for all $m > 0$, which implies that $f(m)$ is increasing in $m$.
Therefore, for $s \ge 3$
\begin{align}
N_{3}(G)
\ge \frac{sn}{2} - \frac{s}{2} \left(4s-4t-4m_{s,t}+n^2-4t_{2}(n)\right)^{1/2} - m_{s,t}
= s\cdot n^{-}_{s,t} - m_{s,t}. \notag
\end{align}
For the case $s=2$, one could easily check that the minimum of $f(m)$ is uniquely attained at $m=m_{s,t}$.
Therefore, if $s_{i}=s$ for some $i\in\{1,2\}$,
then $N_{3}(G) \ge s\cdot n^{-}_{s,t} - m_{s,t}$ for all $s> t \ge 1$.

If $N_{3}(G) = s\cdot n^{-}_{s,t} - m_{s,t}$, then the argument above shows that
we must have $|V_1| =n- n^{-}_{s,t} =  n^{+}_{s,t}$ and $|V_2| = n^{-}_{s,t}$,
all edges in $B$ are contained in $V_1$, all pairs in $M$ must be contained in one potential
copy of $K_3$, and no two pairs in the same potential copy.
Therefore, $G \cong BM_{s,t}(n)$.

\textbf{Case 2:} $s_1 \ge 1$ and $s_2 \ge 1$.

Notice that for every $e \in M$ there are at most two potential copies of $K_3$ containing $e$.
Since $x\ge y$, this gives
\begin{align}
N_{3}(G) \ge s_1 y + s_2 x - 2m
& \ge (s-1) y + x - 2m \notag\\
& = (s-2)y + n -2m \notag\\
& = \frac{sn}{2} - \frac{s-2}{2}\left(4s-4t-4m+n^2-4t_{2}(n)\right)^{1/2} -2m =: g(m). \notag
\end{align}
Let us first assume that $s \ge 20$.
Since
\begin{align}
\frac{\mathrm{d} g(m)}{\mathrm{d} m} = \frac{s-2}{\left(4s-4t-4m+n^2-4t_{2}(n)\right)^{1/2}} -2. \notag
\end{align}
and
\begin{align}
s-2 > 2 \left(4s-4+1\right)^{1/2} \ge 2 \left(4s-4t-4m+n^2-4t_{2}(n)\right)^{1/2}, \notag
\end{align}
$\frac{\mathrm{d} g(m)}{\mathrm{d} m} > 0$ for $m > 0$.
Therefore, by Lemma \ref{lemma-n+n-mst},
\begin{align}
N_{3}(G)
\ge g(m_{s,t})
= f(m_{s,t}) +  \left(4s-4t-4m_{s,t}+n^2-4t_{2}(n)\right)^{1/2} - m_{s,t}
\ge f(m_{s,t}), \notag
\end{align}
and equality holds iff for some $p \in \mathbb{N}$
\begin{align}
s-t =
\begin{cases}
p^2-1, & \text{ if } n \equiv 0 \mod 2, \\
p(p+1)-1, & \text{ if }  n \equiv 1 \mod 2.
\end{cases} \tag{$\star$}
\end{align}
For $s \le 19$ a computer-aided calculation
shows that $f(m_{s,t}) \le \min_{m}\{g(m)\}$
always holds\footnote{A simple Mathematica worksheet verifying this fact
can be found at the web pages   {\tt http://homepages.math.uic.edu/\~{}mubayi/papers/ErdosRademacher.pdf}.}.
Moreover, the minimum of $g(m)$ is uniquely achieved at $m = m_{s,t}$
except for when $(s,t) \in \{(2,1),(3,1),(4,1)\}$ and $n$ even, or $(s,t) \in \{(3,2),(4,1),(5,1),(6,1)\}$ and $n$ odd.

If $N_{3}(G) = s\cdot n^{-}_{s,t} - m_{s,t}$, then the argument above shows that
$(\star)$ holds, $|V_1| =n- n^{-}_{s,t} =  n^{+}_{s,t}$ and $|V_2| = n^{-}_{s,t}$,
exactly one edge $e\in B$ is contained in $V_2$,
all other edges in $B$ are contained in $V_1$,
and all pairs in $M$ must be contained in two potential
copies of $K_3$.
Therefore, $G \cong BS_{s,t}(n)$.

For $(s,t) \in \{(2,1),(3,1),(4,1)\}$ and $n$ even, or $(s,t) \in \{(3,2),(4,1),(5,1),(6,1)\}$ and $n$ odd,
our bound $s\cdot n^{-}_{s,t} - m_{s,t}$
in Theorem \ref{thm-count-triangles} is also tight,
but there are more constructions that achieve this bound.
One could easily recover all these constructions using our calculation file.
\end{proof}

\subsection{Proof of Theorem \ref{thm-count-Kk-s-t=1}}
In this section we prove Theorem \ref{thm-count-Kk-s-t=1}.
Recall that for $n,k \in \mathbb{N}$,
$q_{n,k} = \lf n/(k-1)\rf$ and $r_{n,k} = n-(k-1)q_{n,k}$.

\begin{proof}[Proof of Theorem \ref{thm-count-Kk-s-t=1}]
Let $s \ge 2, k \ge 4$ be fixed integers and $n$ be sufficiently large.
Let $q = q_{n,k}$ and $r = r_{n,k}$.
Let $G$ be a graph on $n$ vertices with $t_{k-1}(n)+s-1$ edges and $\tau_{k}(G) = s$.
Since $N_{k}(n,s)< (s+1/2)\cdot c(n,K_k)$,
we may assume that $N_{k}(G) \le (s+1/2) \cdot c(n,K_k)$.
So by Lemma \ref{lemma-G-G[V1,...Vk-1]-is-matching}, there exists a partition $V(G) = V_1 \cup \cdots \cup V_{k-1}$ such that
$B:= G - G[V_1,\ldots, V_{k-1}]$ is a matching of size $s$.

Let $x_i = |V_i|$ for $i \in [k-1]$ and without loss of generality we may assume that $x_1 \ge \cdots \ge x_{k-1}$.
Let $H = G[V_1,\ldots, V_{k-1}]$, $M = K[V_1,\ldots, V_{k-1}]-H$, and $m = |M|$.
Since $t_{k-1}(n) - 1 = |H| = |K[V_1,\ldots, V_{k-1}]| - m$, we obtain $m\in\{0,1\}$ and
\begin{align}
\sum_{1\le i < j \le k-1}x_ix_j = t_{k-1}(n) - 1 + m. \notag
\end{align}

Suppose that $m=1$.
Then $\sum_{1\le i < j \le k-1}x_ix_j = t_{k-1}(n)$,
so $x_1 = \cdots = x_{r} = q+1$ and $x_{r+1} = \cdots = x_{k-1} = q$.

Let $s_i = |B \cap \binom{V_i}{2}|$ for $i\in [k-1]$
and $S = \{i\in [k-1]: s_{i}\ge 1\}$.

\textbf{Case 1:} $|S| = 1$.

Let $i_0 \in [k-1]$ such that $s_{i_0} = s$.
Then there are $s\cdot \prod_{i\neq i_0}x_i$ potential copies of  $K_k$.
Let $uv \in M$.
If $uv$ has empty intersection with all edges in $B$,
then there are at most $s\cdot n^{k-4} = o(n^{k-3})$ potential copies of $K_k$ containing $uv$.
If $uv$ has nonempty intersection with some $e\in B$,
then every potential copy of $K_k$ that contains $uv$ must contain $e$ as well.
So in this case there are at most $\left(\prod_{i \not\in \{i_0\}} x_{i}\right)/x_{k-1}$ potential copies of $K_{k}$ containing $uv$.
Therefore,
\begin{align}
N_{k}(G)
& \ge s \cdot \prod_{i \not\in \{i_0\}} x_{i} - \frac{1}{x_{k-1}}\prod_{i \not\in \{i_0\}} x_{i} \notag\\
& \ge \left(s-\frac{1}{x_{k-1}}\right)\prod_{i=2}^{k-1}x_{i}
=
\begin{cases}
\left(s-\frac{1}{q}\right) q^{k-2}, &\text{if $r \le 1$}, \\
\left(s-\frac{1}{q}\right) (q+1)^{r-1}q^{k-r-1}, & \text{if $2 \le r \le k-2$},
\end{cases} \notag\\
& \ge N_{k}(n,s),  \notag
\end{align}
and equality holds only if $r=1$.

\textbf{Case 2:} $|S| \ge 2$.

The number of  potential copies of $K_{k}$ is $\sum_{i=1}^{k-1} \left(s_{i}\cdot \prod_{j\neq i}x_{j}\right)$.
Suppose that the pair in $M$ has nonempty intersection with $V_{i_0}$ and $V_{i_1}$ for some $i_0,i_1 \in [k-1]$.
If $s_{i_0} = 0$, then there are at most $\left(\prod_{i \neq i_0} x_{i}\right)/x_{k-1}$ potential copies of
$K_k$ containing the pair in $M$.
If both $s_{i_0} \ge 1$ and $s_{i_1} \ge 1$,
then there are most $2 \prod_{i \neq i_0,i_1} x_{i}$
potential copies of $K_{k}$ containing the pair in $M$.
Therefore,
\begin{align}
N_{k}(G)
\ge \sum_{i=1}^{k-1} \left(s_{i}\cdot \prod_{j\neq i}x_{j}\right) - 2 \prod_{i \neq i_0,i_1} x_{i}
& = \left( \sum_{i=1}^{k-1}\frac{s_i}{x_i} - \frac{2}{x_{i_0}x_{i_1}}\right) \prod_{j=1}^{k-1}x_{j} \notag\\
& \ge \left( \frac{s-2}{x_1} + \frac{1}{x_{i_0}}+ \frac{1}{x_{i_1}} - \frac{2}{x_{i_0}x_{i_1}}\right)
        \prod_{j=1}^{k-1}x_{j} \notag
\end{align}
Since
\begin{align}
\frac{1}{x_{i_0}}+ \frac{1}{x_{i_1}} - \frac{2}{x_{i_0}x_{i_1}}
=\frac{1}{2}- 2 \left(\frac{1}{2}-\frac{1}{x_{i_0}}\right)\left(\frac{1}{2}-\frac{1}{x_{i_1}}\right) \notag
\end{align}
is decreasing in $x_{i_0}$ and $x_{i_1}$,
\begin{align}
\frac{1}{x_{i_0}}+ \frac{1}{x_{i_1}} - \frac{2}{x_{i_0}x_{i_1}}
\ge \frac{1}{x_{1}} + \frac{1}{x_2} - \frac{2}{x_1x_2}. \notag
\end{align}
Therefore,
\begin{align}
N_{k}(G)
 \ge \left( \frac{s-1}{x_1} + \frac{1}{x_2} - \frac{2}{x_1x_2}\right)
        \prod_{j=1}^{k-1}x_{j}
& =
\begin{cases}
\left(s-\frac{2}{q}\right)q^{k-2}, & \text{if $r=0$}, \\
\left(s-\frac{1}{q}\right)q^{k-2}, & \text{if $r=1$}, \\
\left(s-\frac{2}{q+1}\right)(q+1)^{r-1}q^{k-r-1}, & \text{if $2 \le r \le k-2$}.
\end{cases}\notag\\
& \ge N_{k}(n,s). \notag
\end{align}
Note that if $s\ge 3$, then the first inequality above is strict since there are copies of $K_k$ in $G$ containing
at least two edges in $B$.

Now we may assume that $m=0$.
Then every $e \in B$ is contained in at least $\prod_{i=2}^{k-1}x_{i}$ copies of $K_{k}$
and hence
\begin{align}
N_{k}(G) \ge s \cdot \prod_{i=2}^{k-1}x_{i}.\notag
\end{align}
So we just need to find the minimum of $\prod_{i=2}^{k-1}x_{i}$
subject to the constraint that $\prod_{i=1}^{k-1}x_i = t_{k-1}(n)-1$.

If $r = 0$, then $x_1 = q+1, x_2 = \cdots = x_{k-2} = q$, and $x_{k-1} = q-1$.
%\begin{align}
%x_1 = q+1, x_2 = \cdots = x_{k-2} = q, \quad {\rm and} \quad x_{k-1} = q-1. \notag
%\end{align}
Therefore, $\prod_{i=2}^{k-1}x_{i} = q^{k-3}(q-1)$.

If $r = 1$, then $x_1 = x_2 = q+1, x_3 = \cdots = x_{k-2} = q$, and $x_{k-1} = q-1$.
%\begin{align}
%x_1 = x_2 = q+1, x_3 = \cdots = x_{k-2} = q,  \quad {\rm and} \quad x_{k-1} = q-1. \notag
%\end{align}
Therefore, $\prod_{i=2}^{k-1}x_{i} = q^{k-4}(q+1)(q-1)$.

If $r\ge 2$, then
\begin{align}
{\rm either} & \quad x_1 = \cdots = x_{r+1} = q+1, x_{r+2} = \cdots = x_{k-2} = q, x_{k-1}= q-1 \notag\\
{\rm or} & \quad x_1 = q+2, x_2 = \cdots = x_{r-1} = q+1, x_{r} = \cdots = x_{k-1} = q. \notag
\end{align}
The later one gives a smaller $\prod_{i=2}^{k-1}x_{i}$, which is $(q+1)^{r-2}p^{k-r}$.

Therefore, for the case $m=0$
\begin{align}
N_{k}(G)
& \ge
\begin{cases}
s \cdot q^{k-3}(q-1), & \text{if $r=0$}, \\
s \cdot q^{k-4}(q+1)(q-1), & \text{if $r=1$}, \\
s \cdot (q+1)^{r-2}q^{k-r}, & \text{if $2 \le r \le k-2$}.
\end{cases}\notag\\
& \ge N_{k}(n,s), \notag
\end{align}
and equality only if $r\neq 1$.
\end{proof}
%%%%%%%%%%%%%%%%%%%%%%%%%%%%%%%%%%%%%%%%%%%

%%%%%%%%%%%%%%%%%%%%%%%%%%%%%%%%%%%%%%%%%%%%%%%%%%%%%%
\subsection{Proof of Theorem \ref{thm-count-Kk-s-t>1}}
In this section we prove Theorem \ref{thm-count-Kk-s-t>1}.
Recall that for $n,k\in\mathbb{N}$, $q_{n,k} = \lf n/(k-1)\rf$ and $r_{n,k} = n-(k-1)q_{n,k}$.
For $s > t \ge 1, k\ge 3$,
\begin{align}
R_{k}(n,s,t) = \left(\frac{2(k-1)(s-t)}{k-2} + \frac{\left(k-1-r_{n,k}\right)r_{n,k}}{k-2}\right)^{1/2}, \notag
\end{align}
$n^{+}_{k,s,t} = \frac{n+(k-2)R_{k}(n,s,t)}{k-1}$, and $n^{-}_{k,s,t} = \frac{n-R_{k}(n,s,t)}{k-1}$.

\begin{proof}[Proof of Theorem \ref{thm-count-Kk-s-t>1}]
Let $k \ge 4$, $s > t \ge 2$ be fixed integers and $n$ be sufficiently large.
Suppose that $s > 2R_{k}(n,s,t)$.
Let $q = q_{n,k}$, $r = r_{n,k}$, and $R = R_{k}(n,s,t)$.
Let $G$ be a graph on $n$ vertices with $t_{k-1}(n)+t$ edges and $\tau_{k}(G) = s$.
Since $s\cdot \left(n^{-}_{k,s,t}\right) < (s+1/2)\cdot c(n,K_k)$,
we may assume that $N_{k}(G) \le (s+1/2) \cdot c(n,K_k)$.
So by Lemma \ref{lemma-G-G[V1,...Vk-1]-is-matching}, there exists a partition $V(G) = V_1 \cup \cdots \cup V_{k-1}$ such that
$B:= G - G[V_1,\ldots, V_{k-1}]$ is a matching of size $s$.

Let $x_i = |V_i|$ for $i \in [k-1]$ and without loss of generality we may assume that $x_1 \ge \cdots \ge x_{k-1}$.
Let $H = G[V_1,\ldots, V_{k-1}]$, $M = K[V_1,\ldots, V_{k-1}]-H$, and $m = |M|$.
Since $t_{k-1}(n) +t -s = |H| = |K[V_1,\ldots, V_{k-1}]| - m$,
\begin{align}
\sum_{1\le i < j \le k-1}x_ix_j = t_{k-1}(n) + t - s + m, \notag
\end{align}
which is equivalent to
\begin{align}
\sum_{i=1}^{k-1}x_{i}^2 = n^2 - 2 t_{k-1}(n) + 2s -2t -2m. \notag
\end{align}
Let $s_i = |B \cap \binom{V_i}{2}|$ for $i\in [k-1]$
and $S = \{i\in [k-1]: s_{i}\ge 1\}$.

\textbf{Case 1:} $|S| = 1$.

Without loss of generality we may assume that $s_1 = s$ since the other cases can be solved using a similar argument.
Notice that there are $s\cdot\prod_{i=2}^{k-1}x_i$ potential copies of $K_k$,
and for every $e \in M$ there are at most $\prod_{i=2}^{k-2}x_i$ potential copies of $K_k$ containing $e$.
Therefore,
\begin{align}
N_{k}(G)
\ge s\cdot\prod_{i=2}^{k-1}x_i - m \cdot \prod_{i=2}^{k-2}x_i = \left(s-\frac{m}{x_{k-1}}\right)\cdot \prod_{i=2}^{k-1}x_i. \notag
\end{align}
Fix $0 \le m \le s-t$. Let $\mathbb{R}_{\ge 0}$ be the collection of all nonnegative real numbers.
Define
\begin{align}
C_{m}\left(\mathbb{N}\right) = \left\{(x_1,\ldots,x_{k-1})\in \mathbb{N}^{k-1}:
        \sum_{i=1}^{k-1}x_i = n, \sum_{i=1}^{k-1}x_{i}^2 = n^2 - 2 t_{k-1}(n) + 2s -2t -2m\right\}, \notag
\end{align}
and
\begin{align}
C_{m}\left(\mathbb{R}\right) = \left\{(x_1,\ldots,x_{k-1})\in \mathbb{R}_{\ge 0}^{k-1}:
        \sum_{i=1}^{k-1}x_i = n, \sum_{i=1}^{k-1}x_{i}^2 = n^2 - 2 t_{k-1}(n) + 2s -2t -2m\right\}. \notag
\end{align}
Note that $C_{m}\left(\mathbb{N}\right) \subset C_{m}\left(\mathbb{R}\right)$.
In order to get a lower bound for $N_{k}(G)$ we need to solve the following optimization problem.
\begin{align}
{\rm OPT_{m}-A: }
\begin{cases}
{\rm Minimize} & \quad \left(s-\frac{m}{x_{k-1}}\right)\cdot \prod_{i=2}^{k-1}x_i \\
{\rm subject\; to} & \quad (x_1,\ldots,x_{k-1}) \in C_{m}\left(\mathbb{N}\right).
\end{cases} \notag
\end{align}
However, it is not easy to get an optimal solution for ${\rm OPT_m-A}$.
So we are going to consider the following two auxiliary optimization problems.
Let
\begin{align}
{\rm OPT_{m}-B: }
\begin{cases}
{\rm Minimize} & \quad \left(s-\frac{m}{x_{k-1}}\right)\cdot \prod_{i=2}^{k-1}x_i \\
{\rm subject\; to} & \quad (x_1,\ldots,x_{k-1}) \in C_{m}\left(\mathbb{R}\right),
\end{cases} \notag
\end{align}
and
\begin{align}
{\rm OPT_{m}-C:}
\begin{cases}
{\rm Minimize} & \quad \prod_{i=2}^{k-1}x_i \\
{\rm subject\; to} & \quad (x_1,\ldots,x_{k-1}) \in C_{m}\left(\mathbb{R}\right).
\end{cases} \notag
\end{align}

Let ${\rm opt_{m}^a}$, ${\rm opt_{m}^b}$, and ${\rm opt_{m}^c}$ denote the optimal value
of the optimization problems ${\rm OPT_{m}-A}$, ${\rm OPT_{m}-B}$, ${\rm OPT_{m}-C}$, respectively.
It is easy to see that ${\rm opt_{m}^a} \ge {\rm opt_{m}^b}$.
Moreover, if ${\rm OPT_{m}-B}$ has an optimal solution
$x_1,\ldots,x_{k-1}$ such that $x_{i} \in \mathbb{N}$,
then ${\rm opt_{m}^a} = {\rm opt_{m}^b}$.
Our goal is to find ${\rm opt_{m}^b}$ and it will be a lower bound for $N_{k}(G)$.

\begin{claim}\label{claim-opt-a-and-opt-b}
There exists a constant $C > 0$ such that
\begin{align}
\left(s-\frac{k-1}{n}m\right)\cdot {\rm opt_{m}^c} - C n^{k-4}
<  {\rm opt_{m}^b}
\le \left(s-\frac{k-1}{n}m\right)\cdot {\rm opt_{m}^c} + C n^{k-4}.\notag
\end{align}
\end{claim}
\begin{proof}[Proof of Claim \ref{claim-opt-a-and-opt-b}]
We abuse notation by assuming that $x_1,\ldots, x_{k-1}$ is an optimal solution of ${\rm OPT_{m}-B}$.
Since $\sum_{1\le i < j \le k-1}x_{i}x_{j} = t_{k-1}(n) + t - s + m > t_{k-1}(n)-s$,
by Lemma \ref{lemma-size-Vi}, $n/(k-1)-s \le x_{i} \le n/(k-1)+s$ for all $i\in [k-1]$.
Therefore,
\begin{align}
{\rm opt_{m}^b}
= \left(s-\frac{m}{x_{k-1}}\right)\cdot \prod_{i=2}^{k-1}x_i
& \ge \left(s-\frac{m}{n/(k-1)-s}\right)\cdot \prod_{i=2}^{k-1}x_i \notag\\
& = \left(s-\frac{(k-1)m}{n}\right) \cdot \prod_{i=2}^{k-1}x_i
    - \frac{(k-1)^2s m}{n(n-ks+s)} \cdot \prod_{i=2}^{k-1}x_i  \notag\\
& > \left(s-\frac{(k-1)m}{n}\right) \cdot \prod_{i=2}^{k-1}x_i - C n^{k-4} \notag\\
& \ge \left(s-\frac{(k-1)m}{n}\right) \cdot {\rm opt_{m}^c} - C n^{k-4}. \notag
\end{align}

Now let $x_1',\ldots,x_{k-1}'$ be an optimal solution of ${\rm OPT_{m}-C}$.
Then similarly we have
\begin{align}
\left(s-\frac{(k-1)m}{n}\right) \cdot {\rm opt_{m}^c}
& = \left(s-\frac{m}{n/(k-1)}\right) \cdot \prod_{i=2}^{k-1}x'_i \notag\\
& \ge \left(s-\frac{m}{x'_{k-1}-s}\right)\cdot \prod_{i=2}^{k-1}x'_i \notag\\
& = \left(s-\frac{m}{x'_{k-1}}\right)\cdot \prod_{i=2}^{k-1}x'_i - \frac{sm}{x'_{k-1}(x'_{k-1}+s)}\prod_{i=2}^{k-1}x'_i
 \ge {\rm opt_{m}^b} - Cn^{k-4}. \notag
\end{align}
\end{proof}

Claim \ref{claim-opt-a-and-opt-b} says that one could view ${\rm opt_{m}^c}$ as a "trajectory" for ${\rm opt_{m}^b}$,
and we will use it to show that ${\rm opt_{m-1}^b} \le {\rm opt_{m}^b}$.
Let us solve the optimization problem ${\rm OPT_{m}-C}$ first.
We use the Lagrangian multiplier method.
Let
\begin{align}
\mathcal{L}(\vec{x},\lambda,\mu)  = \prod_{i=2}^{k-1}x_i
                                + \lambda \left(\sum_{i=1}^{k-1}x_i - n\right)
                                + \mu \left(\sum_{i=1}^{k-1}x_{i}^2 -\left( n^2 - 2 t_{k-1}(n) + 2s -2t -2m \right)\right). \notag
\end{align}
Again, we abuse notation here by assuming that $(x_1,\ldots,x_{k-1}) \in C_{m}(\mathbb{R})$
is an optimal solution of ${\rm OPT_{m}-C}$.
Then by the Lagrangian multiplier method,
\begin{align}
\begin{cases}
\frac{\partial \mathcal{L}}{\partial x_1} & =  \lambda + 2\mu x_{1}  = 0  \Rightarrow x_1 = - \frac{\lambda}{2\mu}, \\
\frac{\partial \mathcal{L}}{\partial x_j} & =  \frac{\prod_{i=2}^{k-1}x_i}{x_j} + \lambda + 2\mu x_{j}  = 0, \\
\frac{\partial \mathcal{L}}{\partial \lambda} & = \sum_{i=1}^{k-1}x_i - n =0, \\
\frac{\partial \mathcal{L}}{\partial \mu} & = \sum_{i=1}^{k-1}x_{i}^2 -\left(n^2 - 2 t_{k-1}(n) + 2s -2t -2m \right) =0.
\end{cases} \notag
\end{align}
Let $\pi = \prod_{i=2}^{k-1}x_i$.
Note that the equation
\begin{align}
\frac{\pi}{x} + \lambda + 2\mu x = 0 \notag
\end{align}
has only two solutions
\begin{align}
x'  = \frac{-\lambda + \sqrt{\lambda^2 - 8\mu  \pi}}{4\mu}
\quad {\rm and} \quad
x''  = \frac{-\lambda - \sqrt{\lambda^2 - 8\mu  \pi}}{4\mu}. \notag
\end{align}
Therefore, for every $2 \le i \le k-1$ either $x_i = x'$ or $x_i = x''$.

\begin{claim}\label{claim-x2=x3=...=xk-1}
$x_1 \ge x_2 = \cdots = x_{k-1}$.
\end{claim}
\begin{proof}[Proof of Claim \ref{claim-x2=x3=...=xk-1}]
First we show that $x_1 \ge x_i$ for all $2 \le i \le k-1$.
Suppose to the contrary that there exists some $i \in [k-1]\setminus\{1\}$
such that $x_i > x_1$, and without loss of generality we may assume that $x_2 > x_1$.
Then let $x'_i = x_{i}$ for $3 \le i \le k-1$, $x_1' = x_2$, and $x'_2 = x_1$.
It is clear that $(x_1',\ldots,x_{k-1}') \in C_{m}(\mathbb{R})$,
but $\prod_{i=2}^{k-1}x_{i}'< \prod_{i=2}^{k-1}x_{i}$, which contradicts our assumption that
$(x_1,\ldots,x_{k-1})$ is an optimal solution of ${\rm OPT_{m}-C}$.
Therefore,  $x_1 \ge x_i$ for all $2 \le i \le k-1$.

Now we show that $x_2 = \cdots = x_{k-1}$.
Suppose that $x_{i_1} \neq x_{i_2}$ for some $2 \le i_1 < i_2 \le k-1$.
Then $\{x_{i_1}, x_{i_2}\} = \{x',x''\}$, which implies that $x_{i_1}+x_{i_2} = -\lambda/(2\mu) = x_1$.
Since $\sum_{1\le i < j \le k-1}x_ix_j = t_{k-1}(n)+t-s+m > t_{k-1}(n)-s$,
by Lemma \ref{lemma-size-Vi},
$|x_i - n/(k-1)|<s$ for all $i\in[k-1]$.
Therefore,
\begin{align}
x_{i_1}+x_{i_2}
> 2\times \frac{n}{k-1} - 2s
> \frac{n}{k-1} + s
> x_1, \notag
\end{align}
a contradiction.
Therefore, $x_2 = \cdots = x_{k-1}$.
\end{proof}

By Lemma \ref{lemma-expression-turan-number},
\begin{align}
n^2 - 2t_{k-1}(n)
= \frac{n^2}{k-1} + \frac{(k-1-r)r}{k-1}. \notag
\end{align}
Let $x=x_1$, $y = x_2 = \cdots = x_{k-1}$.
Since $(x_1,\ldots,x_{k-1}) \in C_{m}(\mathbb{R})$,
\begin{align}
\begin{cases}
x + (k-2)y = n, \\
x^2 +(k-2)y^2 = \frac{n^2}{k-1} + \frac{(k-1-r)r}{k-1} + 2s -2t -2m, \\
x_i \ge 0, \forall i \in [k-1],
\end{cases}  \notag
\end{align}
which implies
\begin{align}
\begin{cases}
x = \frac{n}{k-1} + (k-2) \Delta_m \\
y = \frac{n}{k-1} - \Delta_m,
\end{cases} \notag
\end{align}
where
\begin{align}
\Delta_m:= \frac{\left(2(k-1)(k-2)(s-t-m)+(k-2)(k-1-r)r\right)^{1/2}}{(k-1)(k-2)}. \notag
\end{align}
Therefore,
\begin{align}
{\rm opt_{m}^{c}}
= y^{k-2}
= \left(\frac{n}{k-1} - \Delta_{m}\right)^{k-2}. \notag
\end{align}

Now we are going to use ${\rm opt_{m}^{c}}$ to describe the behavior of ${\rm opt_{m}^{b}}$.

\begin{claim}\label{claim-best-optb}
The value ${\rm opt_{m}^{b}}$ is strictly increasing in $m$.
In particular,
${\rm opt_{0}^{b}} < {\rm opt_{m}^{b}}$ for all $m>0$.
\end{claim}
\begin{proof}[Proof of Claim \ref{claim-best-optb}]
Since
\begin{align}
{\rm opt_{m}^{c}}
 = \left(\frac{n}{k-1} - \Delta_{m}\right)^{k-2}
 =  \left(\frac{n}{k-1}\right)^{k-2} - (k-2)\Delta_{m}\left(\frac{n}{k-1}\right)^{k-3} + \Theta(n^{k-4}), \notag
\end{align}
by Claim \ref{claim-opt-a-and-opt-b}, there exists a constant $C>0$ such that
\begin{align}
{\rm opt_{m}^{b}}
& =  \left(s-\frac{k-1}{n}m\right)\cdot {\rm opt_{m}^{c}} \pm Cn^{k-4} \notag\\
& = s  \left(\frac{n}{k-1}\right)^{k-2} - \left(m + s(k-2)\Delta_{m}\right) \left(\frac{n}{k-1}\right)^{k-3} \pm C'n^{k-4}. \notag
\end{align}
Therefore,
\begin{align}
{\rm opt_{m-1}^{b}} - {\rm opt_{m}^{b}}
=  \left(1 - s(k-2)\left(\Delta_{m-1}-\Delta_{m}\right)\right)\left(\frac{n}{k-1}\right)^{k-3} \pm 2C'n^{k-4}. \notag
\end{align}
Now view $\Delta_m$ as a function of the variable $m$.
Then it is easy to see that $\Delta_m$ is concave down, i.e. $\mathrm{d}^2 \Delta_{m}/\mathrm{d} m^2<0$
for $0 \le m \le s-t$.
Therefore,
\begin{align}
s(k-2)\left(\Delta_{m-1}-\Delta_{m}\right)
& \ge s(k-2) (-1)\cdot\left.\frac{\mathrm{d} \Delta_{m}}{\mathrm{d} m}\right|_{m=0} \notag\\
& =  \frac{s(k-2)}{\left(2(k-1)(k-2)(s-t)+(k-2)(k-1-r)r\right)^{1/2}}. \notag
\end{align}
Since
\begin{align}
s > 2R = 2\frac{\left(2(k-1)(k-2)(s-t)+(k-2)\left(k-1-r\right)r\right)^{1/2}}{k-2}, \notag
\end{align}
we obtain $s(k-2)\left(\Delta_{m-1}-\Delta_{m}\right) > 2$.
Therefore,
\begin{align}
1 - s(k-2)\left(\Delta_{m-1}-\Delta_{m}\right) < -1, \tag{$\ast$}
\end{align}
and hence ${\rm opt_{m-1}^{b}} - {\rm opt_{m}^{b}} <  - \left(n/(k-1)\right)^{k-3} + \Theta(n^{k-4}) < 0$.
\end{proof}
Therefore,
\begin{align}
N_{k}(G)
\ge {\rm opt_{m}^{a}} \ge {\rm opt_{m}^{b}} \ge {\rm opt_{0}^{b}}
= s \cdot {\rm opt_{0}^{c}}
= s \cdot \left(\frac{n}{k-1} - \Delta_0\right)^{k-2}
= s \cdot \left(n^{-}_{k,s,t}\right)^{k-2}. \notag
\end{align}
Here we used that fact that $\Delta_0 = R/(k-1)$.

\textbf{Case 2:} $|S| \ge 2$.

The number of potential copies of $K_{k}$ is $\sum_{i=1}^{k-1} \left(s_{i}\cdot \prod_{j\neq i}x_{j}\right)$.
Suppose that $uv \in M$ satisfies $u \in V_{i_0}$ and $v \in V_{i_1}$ for some $i_0,i_1 \in [k-1]$.
Similar to the proof of Theorem \ref{thm-count-Kk-s-t>1} we may assume that $s_{i_0} \ge 1$ and $s_{i_1} \ge 1$.
Then there are at most $\prod_{i\neq i_{0},i_{1}}x_i$ potential copies of $K_k$ containing $uv$.
Therefore,
\begin{align}
N_{k}(G)
 \ge \sum_{i=1}^{k-1}\left(s_i \cdot \prod_{j\neq i}x_{j}\right)
    - 2\sum_{uv\in M} \prod_{i\neq i_{0},i_{1}\atop u\in V_{i_0},v\in V_{i_1}}x_i
 = \left(\sum_{i=1}^{k-1}\frac{s_i}{x_i} - \sum_{uv\in M\atop u\in V_{i_0},v\in V_{i_1}}\frac{2}{x_{i_0}x_{i_1}}\right)
     \prod_{i=1}^{k-1}x_{i}. \notag
\end{align}
We abuse notation by assuming that $x_{i_0}x_{i_1} = \min\{x_ix_j: \exists uv\in M \text{ such that } u\in V_i, v\in V_j\}$.
Then
\begin{align}
N_{k}(G)
 \ge \left(\sum_{i=1}^{k-1}\frac{s_i}{x_i} - \frac{2m}{x_{i_0}x_{i_1}}\right)
     \prod_{i=1}^{k-1}x_{i}
 = \left(\frac{s-2}{x_1} + \frac{1}{x_{i_0}}+\frac{1}{x_{i_1}}- \frac{2m}{x_{i_0}x_{i_1}}\right)
     \prod_{i=1}^{k-1}x_{i}. \notag
\end{align}
Since
\begin{align}
\frac{1}{x_{i_0}}+\frac{1}{x_{i_1}}- \frac{2m}{x_{i_0}x_{i_1}}
= \frac{1}{2m} - 2m \left(\frac{1}{2m}-\frac{1}{x_{i_0}}\right)\left(\frac{1}{2m}-\frac{1}{x_{i_1}}\right), \notag
\end{align}
is decreasing in $x_{i_0}$ and $x_{i_1}$,
\begin{align}
\frac{1}{x_{i_0}}+\frac{1}{x_{i_1}}- \frac{2m}{x_{i_0}x_{i_1}}
\ge \frac{1}{x_1} + \frac{1}{x_{2}} - \frac{2m}{x_1x_2}. \notag
\end{align}
Therefore,
\begin{align}
N_{k}(G)
& \ge \left(\frac{s-1}{x_1} + \frac{1}{x_{2}} - \frac{2m}{x_{1}x_{2}}\right)
     \prod_{i=1}^{k-1}x_{i} \notag\\
& = \left(\frac{s}{x_1} + \frac{x_1-x_2}{x_1x_{2}} - \frac{2m}{x_{1}x_{2}}\right)
     \prod_{i=1}^{k-1}x_{i}
  = \left(s + \frac{x_1-x_2}{x_{2}} - \frac{2m}{x_{2}}\right)
     \prod_{i=2}^{k-1}x_{i}. \notag
\end{align}
Therefore, in order to get a lower bound for $N_{k}(G)$ we need solve the following optimization problem.
\begin{align}
{\rm OPT_{m}-D: }
\begin{cases}
{\rm Minimize} & \quad \left(s + \frac{x_1-x_2}{x_{2}} - \frac{2m}{x_{2}}\right)\prod_{i=2}^{k-1}x_{i} \\
{\rm subject \; to} & \quad (x_1,\ldots,x_{k-1}) \in C_{m}(\mathbb{N}).
\end{cases} \notag
\end{align}
Similarly, we are going to consider the following auxiliary optimization problem.
\begin{align}
{\rm OPT_{m}-E: }
\begin{cases}
{\rm Minimize} & \quad \left(s + \frac{x_1-x_2}{x_{2}} - \frac{2m}{x_{2}}\right)\prod_{i=2}^{k-1}x_{i} \\
{\rm subject \; to} & \quad (x_1,\ldots,x_{k-1}) \in C_{m}(\mathbb{R}).
\end{cases} \notag
\end{align}
Theoretically, one could solve ${\rm OPT_{m}-E}$ exactly using the Lagrange multiplier method.
However, the optimal solution of ${\rm OPT_{m}-E}$ is very complicated.
So we are going to compare ${\rm OPT_{m}-E}$ with ${\rm OPT_{m}-C}$.

Let ${\rm opt_{m}^d}$ and ${\rm opt_{m}^e}$ denote the optimal values
of the optimization problems ${\rm OPT_{m}-D}$ and ${\rm OPT_{m}-E}$, respectively.
It is easy to see that ${\rm opt_{m}^d} \ge {\rm opt_{m}^e}$.
%and when $x_1 > x_2$
%(actually, one could show that $x_1>x_2$ using the Lagrangian multiplier method)
%the inequality is strict.
The following claim is very similar to  Claim~\ref{claim-opt-a-and-opt-b},
and can be proved in a similar fashion so so we omit the proof.
\begin{claim}\label{claim-opt-e-and-opt-c}
There exists a constant $C > 0$ such that
\begin{align}
\left(s-\frac{2(k-1)m}{n}\right)\cdot {\rm opt_{m}^c} - C n^{k-4}
<  {\rm opt_{m}^e}
\le \left(s-\frac{2(k-1)m}{n}\right)\cdot {\rm opt_{m}^c} + C n^{k-4}.\notag
\end{align}
\end{claim}
\medskip

\begin{claim}\label{claim-best-opte}
The value ${\rm opt_{m}^{e}}$ is strictly increasing in $m$.
In particular,
${\rm opt_{0}^{e}} < {\rm opt_{m}^{e}}$ for all $m>0$.
\end{claim}
\begin{proof}[Proof of Claim \ref{claim-best-opte}]
The proof is basically the same as the proof for Claim \ref{claim-best-optb}.
The only difference is that $s>2R$ implies that
there exists $\varepsilon>0$ such that
$s(k-2)\left(\Delta_{m-1}-\Delta_{m}\right) > 2+\varepsilon$.
Therefore,
$(\ast)$ now becomes
$$2 - s(k-2)\left(\Delta_{m-1}-\Delta_{m}\right) < - \varepsilon,$$
which implies that
\begin{align}
{\rm opt_{m-1}^{b}} - {\rm opt_{m}^{b}}
& =  \left(2- s(k-2)\left(\Delta_{m-1}-\Delta_{m}\right)\right)\left(\frac{n}{k-1}\right)^{k-3} \pm 2C'n^{k-4} \notag\\
& < - \varepsilon \left(n/(k-1)\right)^{k-3}
+ \Theta(n^{k-4}) < 0. \notag
\end{align}
\end{proof}

Therefore, if $s > 2R$, then
\begin{align}
N_{k}(G)
\ge {\rm opt_{m}^{d}} \ge {\rm opt_{m}^{e}} \ge {\rm opt_{0}^{e}}
\ge s \cdot {\rm opt_{0}^{c}}
= s \cdot \left(\frac{n}{k-1} - \Delta_{0}\right)^{k-2}
= s \cdot \left(n^{-}_{k,s,t}\right)^{k-2}. \notag
\end{align}
Note that we may assume that $s-t \ge 2$ since the case $s-t=1$ has been solved by Theorem \ref{thm-count-Kk-s-t=1}.
Therefore, there exists copies of $K_k$ in $G$ that contains at least two edges in $B$,
which implies that the first inequality above is strict.
\end{proof}

%%%%%%%%%%%%%%%%%%%%%%%%%%%%%%%%%%%%%%%%%%
\subsection{Proof of Theorem \ref{thm-count-critical-graphs}}
In this section we prove Theorem \ref{thm-count-critical-graphs}.
We need the following lemma.

\begin{lemma}[\cite{DM10}]\label{lemma-dhruv-critical-bound-tail}
Fix $k\ge 3$ and a $k$-critical graph with $f$ vertices.
Then there are positive constant $\alpha_{F}$ and $\beta_{F}$ such that
if $n$ is sufficiently large,
then $|c(n,F) - \alpha_{F}n^{f-2}| < \beta_{F}n^{f-3}$.
\end{lemma}

Now we are ready to prove Theorem \ref{thm-count-critical-graphs}.

\begin{proof}[Proof of Theorem \ref{thm-count-critical-graphs}]
Let $s> t \ge 1, k \ge 3$ be fixed integers and let $F$ be a $k$-critical graph on $f$ vertices.
Let $n$ be sufficiently large.
Let $G$ be a graph on $n$ vertices with $t_{k-1}(n)+t$ edges and $\tau_{F}(G) = s$.
We may assume that $N_{F}(G) \le s\cdot c(n,F)$, since otherwise we are done.

By Lemma \ref{lemma-G-G[V1,...Vk-1]-is-matching}, there exists a partition $V(G) = V_1 \cup \cdots \cup V_{k-1}$ such that
$B:= G - G[V_1,\ldots, V_{k-1}]$ is a matching of size $s$.
Let $x_i = |V_i|$ for $i \in [k-1]$ and without loss of generality we may assume that $x_1 \ge \cdots \ge x_{k-1}$.
Let $H = G[V_1,\ldots, V_{k-1}]$, $M = K[V_1,\ldots, V_{k-1}]-H$, and $m = |M|$.
Since $t_{k-1}(n) - t = |H| = |K[V_1,\ldots, V_{k-1}]| - m$,
\begin{align}
\sum_{1\le i < j \le k-1}x_ix_j = t_{k-1}(n) + t - s + m. \notag
\end{align}
Therefore, by Lemma \ref{lemma-size-Vi}, $n/(k-1) - s < x_i < n/(k-1)+s$ for all $i\in [k-1]$.
Let
\begin{align}
c_{min} = \min\{c(x_{\sigma(1)},\ldots, x_{\sigma(k-1)}): \sigma\in S_{k-1}\}, \notag
\end{align}
where $S_{k-1}$ is the collection of all permutations of $[k-1]$.
By Lemma \ref{lemma-dhruv-c(n1,...,nk-1,F)},
$c_{min} \ge c(n,F) - \gamma_{F} s n^{f-3}$ for some constant $\gamma_{F}$.
Note the the number of potential copies of $K_k$ is at least $s \cdot c_{min}$.
Since every $e \in M$ is contained in at most $n^{f-3}$ potential copies of $K_k$,
\begin{align}
N_{F}(G)
\ge s \cdot c_{min} - m n^{f-3}
\ge s \cdot c(n,F) - Cn^{f-3} \notag
\end{align}
for some constant $C$.
This completes the proof of Theorem \ref{thm-count-critical-graphs}.
\end{proof}
%%%%%%%%%%%%%%%%%%%%%%%%%%%%%%%%%%%%%%%%%%%%%%%%%

%%%%%%%%%%%%%%%%%%%%%%%%%%%%%%%%%%%%%%%%%%%%%%%%%%%
\section{Concluding remarks}
We proved several bounds on the number of copies of $K_k$ (and also for $k$-critical graphs $F$) in
a graph $G$ on $n$ vertices with $t_{k-1}(n)+t$ edges and $\tau_{k}(G) = s$.
In our proof we need $s$ and $t$ to be fixed.
Using the same method we are able to show that the same conclusions as in
Theorems \ref{thm-count-triangles}, \ref{thm-count-Kk-s-t=1},
\ref{thm-count-Kk-s-t>1}, and \ref{thm-count-critical-graphs} hold for all $s> t \ge 1$
(for Theorem \ref{thm-count-Kk-s-t>1} we still need $s > 2 R_{k}(n,s,t)$) as long as
$s(s-t)^{1/2} < \xi n$ for some small constant $\xi>0$.
In particular, if $s-t<C$ for some constant $C$,
then the conclusions hold for all $s < \xi' n$ for some small constant $\xi'>0$.
 The proofs are more involved and tedious,
so we chose to omit them here.

%%%%%%%%%%%%%%%%%%%%%%%%%%%%%%%%%%%%%%%%%%%%%%%%%
%\bibliographystyle{elsarticle-num}
\bibliographystyle{abbrv}
\bibliography{count_clique}
\end{document}